\documentclass[11pt,a4paper,reqno]{amsart}
\title[q-deformed]{Topological model for q-deformed rational number and categorification}
\author{Li Fan}
\address{Fl: Department of Mathematical Sciences, Tsinghua University, 100084 Beijing, China.}
\email{fan-l17@tsinghua.org.cn}
\author{Yu Qiu}
\address{Qy: Yau Mathematical Sciences Center and Department of Mathematical Sciences, Tsinghua University, 100084 Beijing, China. \&  Beijing Institute of Mathematical Sciences and Applications, Yanqi Lake, Beijing, China}
\email{yu.qiu@bath.edu}
\usepackage{amsfonts}
\usepackage{amsmath,amssymb,amsthm,amsxtra}
\usepackage{float}
\usepackage[colorlinks, linkcolor=blue!72,anchorcolor=orange,
    citecolor=red,urlcolor=Emerald, bookmarksopen,bookmarksdepth=2]{hyperref}
\usepackage[usenames,dvipsnames]{xcolor}
\usepackage{enumitem}
\usepackage{geometry,array}
\usepackage{graphicx}
\usepackage{subfigure}
\usepackage{bookmark}
\usepackage{tikz}
\usepackage{url}

\usetikzlibrary{matrix,positioning,decorations.markings,arrows,decorations.pathmorphing,	
    backgrounds,fit,positioning,shapes.symbols,chains,shadings,fadings,calc}
\tikzset{->-/.style={decoration={  markings,  mark=at position #1 with
    {\arrow{>}}},postaction={decorate}}}
\tikzset{-<-/.style={decoration={  markings,  mark=at position #1 with
    {\arrow{<}}},postaction={decorate}}}
\usepackage[all]{xy}
\usetikzlibrary{cd}
\usepackage{cleveref}

\theoremstyle{plain}
\newtheorem{theorem}{Theorem}[section]

\newtheorem{lemma}[theorem]{Lemma}
\newtheorem{corollary}[theorem]{Corollary}
\newtheorem{proposition}[theorem]{Proposition}
\newtheorem{propdef}[theorem]{Proposition/Definition}
\newtheorem{lemdef}[theorem]{Lemma/Definition}

\theoremstyle{definition}
\newtheorem{definition}[theorem]{Definition}

\newtheorem{example}[theorem]{Example}

\newtheorem{remark}[theorem]{Remark}
\newtheorem{notations}[theorem]{Notations}
\numberwithin{equation}{section}

\newtheorem{construction}[theorem]{Construction}

\newtheorem{setting}[theorem]{Setting}
\newtheorem*{con}{Conventions}

\def\hh{\mathcal}

\def\<{\langle}
\def\>{\rangle}
\def\={\simeq}
\def\s{\simeq_s}
\def\to{\rightarrow}

\def\NN{\mathbb{N}}
\def\ZZ{\mathbb{Z}}
\def\QQ{\mathbb{Q}}
\def\CC{\mathbb{C}}
\def\XX{\mathbb{X}}
\def\PP{\mathbb{P}}
\def\RR{\mathbb{R}}

\renewcommand{\k}{\mathbf{k}}

\newcommand{\D}{\operatorname{\hh{D}}}

\def\dim{\operatorname{dim}}

\def\Aut{\operatorname{Aut}}

\def\Hom{\operatorname{Hom}}

\def\Ext{\operatorname{Ext}}

\def\ind{\operatorname{ind}}

\def\dind{\ind^{\ZZ^2}}
\newcommand{\sind}{\operatorname{ind}^\XX}

\def\Br{\operatorname{Br}}

\def\PSL{\operatorname{PSL}_2(\ZZ)}

\newcommand{\Cone}{\operatorname{Cone}}

\newcommand{\A}{\operatorname{\hh{A}}}

\newcommand{\pvd}{\operatorname{pvd}}

\newcommand\Sph{\operatorname{Sph}}

\newcommand{\Tri}{\bigtriangleup}

\def\SSo{\mathbf{S}_\Tri}                       
\newcommand{\ST}{\operatorname{ST}}        
\newcommand{\BT}{\operatorname{BT}}        
\newcommand{\MCG}{\operatorname{MCG}}
\newcommand{\Int}{\operatorname{Int}}

\newcommand\ho[1]{\operatorname{H}_{#1}}

\def\surf{\mathbf{S}}                       
\def\M{\mathbf{M}}
\def\bdy{\partial\mathbf{S}}
\def\isurf{\surf^\circ}
\def\surfo{{\mathbf{S}}_\Tri}
\def\surfoi{\surfo^\circ}

\def\DMS{\surfo^\Lambda}

\def\PTSO{\mathbb{P}T(\surf\setminus\Tri)}
\def\PTSt{\mathbb{P}T\SSo}
\def\wPTSt{\widehat{\PTSt}}
\def\PTSt{\mathbb{P}T\SSo}
\newcommand{\PTSo}[1]{\mathbb{P}{T_{#1}}\surfo}
\def\CA{\operatorname{CA}}
\def\wXCA{\widehat{\CA}}      
\newcommand{\qqInt}{\operatorname{Int}^{\q}}
\def\ii{\varrho} 
\def\sii{\varsigma} 

\def\A{{\mathbf{A}}}

\def\Li{\Lambda_1} 
\def\bc{\Sigma_\Delta}
\def\ubc{\widetilde{\Sigma_\Delta}}

\def\cpi{\pi}

\newcommand{\qdH}{\dim^{\q}\Hom^{\ZZ^2}}    
\def\occq{\operatorname{occ}_{q}}
\def\uhom{\overline{\operatorname{hom}}_q}

\def\Xp{X_{\frac{p}{q}}}
\def\Xq{X_{\frac{r}{s}}}
\def\Xqq{X_{\frac{u}{v}}}
\def\Xqs{X_{\frac{p}{q}\oplus\frac{u}{v}}}
\def\Xze{X_{\zero}}
\def\Xin{X_{\infi}}
\def\Xon{X_{\frac{1}{1}}}
\def\qSph{\widehat{\Sph}}

\def\dy{v_1}   
\def\de{v_2}
\def\ds{v_3}
\def\di{z_{\infi}}
\def\dii{z_*}
\def\diii{z_{\zero}}
\def\deci{z_{\infi}}
\def\decs{z_*}
\def\decz{z_{\zero}}

\def\d{z}
\def\pn{p}

\def\pt{z}

\def\disk{\mathbf{D}_{3}}
\def\diski{\mathbf{D}_{3}^\circ}

\def\DT{\operatorname{DT}}
\def\Z{\operatorname{Z}}
\newcommand{\FG}{\operatorname{FG}}
\newcommand{\FGv}{\FG_0}
\def\ue{\eta}
\def\us{\sigma}
\def\ut{\tau}
\def\ug{\gamma}
\def\ua{\alpha}

\def\wg{\widehat{\gamma}}

\def\lwa{\widehat{\alpha}}

\def\lwe{\widehat{\eta}}

\def\lwg{\widehat{\gamma}}
\def\lws{\widehat{\sigma}}
\def\lwt{\widehat{\tau}}
\def\lwth{\widehat{\theta}}
\def\lc{\widehat{c}}
\def\lwq{\lwe_{\frac{r}{s}}}
\def\lwp{\lwe_{\frac{p}{q}}}
\def\lwqq{\lwe_{\frac{u}{v}}}
\def\lwqs{\lwe_{\frac{p}{q}\oplus \frac{u}{v}}}
\def\lwez{\lwe_{\zero}}
\def\lwei{\lwe_{\infi}}
\def\wgz{\wg_{\zero}}
\def\wgi{\wg_{\infi}}

\def\lwte{\widehat{\ut}\wedge\widehat{\ue}}

\def\txq{$\q$}
\def\rat{\frac{r}{s}}
\def\q{\mathfrak{q}}               
\def\qz{q_1}
\def\qx{q_2}
\def\qd{\q}
\def\qid{(\q^{-1})}
\def\qr{\operatorname{\textbf{R}}_{\q}}
\def\qs{\textbf{S}_{\q}}
\def\qrb#1{\overline{\textbf{R}}_{\q}^{\flat}(#1)}
\def\qsb#1{\overline{\textbf{S}}_{\q}^{\flat}(#1)}

\def\uQ{\overline{\QQ}}
\def\upQ{\overline{\QQ^{\geq 0}}}
\def\ugQ{\QQ^+}
\def\Ti{T}
\def\gi{t_1}
\def\gii{t_2}
\def\PSLq{\operatorname{PSL}_{2,\q}(\ZZ)}
\def\gqi{t_{1,\q}}
\def\gqii{t_{2,\q}}
\def\zero{0}
\def\infi{\infty}
\def\pq{p/q}
\def\rs{\frac{r}{s}}
\def\rsi{u/v}

\def\Gro{\operatorname{K}}

\tikzcdset{arrow style=tikz, diagrams={>=stealth}}
\def\nn{node{$\bullet$}}            
\def\ww{node[white]{\tiny{$\bullet$}}node{\tiny{$\circ$}}}
\begin{document}
\maketitle

\def\Qy#1{\textcolor{blue}{Qy: #1}}
\setlength\parindent{0em}
\setlength{\parskip}{5pt}

\def\DGA{\Gamma_\XX A_2}

\newcommand{\cat}[1]{\D_{#1}(A_2)}
\newcommand{\bcat}[1]{\overline{\D}_{#1}(A_2)}


\begin{abstract}
Let $\disk$ be a bigraded 3-decorated disk with an arc system $\A$.
We associate a bigraded simple closed arc $\lwe_{\frac{r}{s}}$ on $\disk$
to any rational number $\frac{r}{s}\in\uQ=\QQ\cup\{\infty\}$.
We show that the right (resp. left) $q$-deformed rational numbers associated to $\frac{r}{s}$,
in the sense of \cite{MGO20} (resp. \cite{BBL}) can be naturally calculated by
the \txq-intersection between $\lwe_{\frac{r}{s}}$ and $\A$ (resp. dual arc system $\A^*$).
The Jones polynomials of rational knots can be also given by such intersections.
Moreover, the categorification of $\lwe_{\frac{r}{s}}$ is given by the spherical object $X_{\frac{r}{s}}$ in the Calabi-Yau-$\XX$ category of Ginzburg dga of type $A_2$. Reduce to CY-2 case, we recover result of \cite{BBL} with a slight improvement.
\end{abstract}
\section{Introductions}
The notion of (right) $q$-deformed rational numbers $[\frac{r}{s}]^{\sharp}$ was originally introduced by Morier-Genoud and Ovsienko in \cite{MGO20} via continued fractions.
They also developed $q$-deformations to irrational numbers in \cite{MGO19} by the convergency property.
Such $q$-deformations own many good combinatorial properties and are related to a wide variety of areas, such as the Farey triangulation, $F$-polynomials of cluster algebras and the Jones polynomial of rational (two-bridge) knots \cite{MGO20}.
Motivated by the study of compactification of spaces of stability conditions,
Bapat, Becker and Licata \cite{BBL} introduced a twin notion, the left $q$-deformation $[\frac{r}{s}]^{\flat}$, which also shares all the good properties of $[\frac{r}{s}]^{\sharp}$. They showed that the two $q$-deformations can be both described via the action of $\PSLq$ by fractional linear transformations. Moreover, Farey graph plays an important role in the definition of $q$-deformations, where the edges are assigned weights according to some iterative rules \cite{MGO20}.

On the other hand, the homotopy classes of simple closed curves on torus with at most one boundary can be parameterized by $\uQ=\QQ\cup\{\infi\}$. We aim to give a topological realization of $q$-deformations and their categorification.
The topological model we use is the decorated surface $\surfo$ with bigrading introduced by Khovanov and Seidel in \cite{KS}.
The bigrading of arcs provides bi-indices for their intersections, which we call $\q$-intersections.
We consider the $A_2$ case, where $\surfo=\disk$ is a disk with three decorations and
the set of simple closed arcs on $\disk$ can be parameterized by $\uQ$.
We show that the right/left $q$-deformed rationals can be naturally calculated by the \txq-intersections between
corresponding arcs (\Cref{thm:leftint} and \Cref{thm:rightint}).
The topological realization directly implies many combinatorial properties of $q$-deformations, including positivity and specialization (\Cref{cor:comp}). 
Surprisingly, the bi-index always collapses into one, 
which is not obvious from the construction/definition of $\q$-intersection.

For the categorification, we consider the Calabi-Yau-$\XX$ category $\D_{\XX}(\surfo)$ associated to $\surfo$ (cf. \cite{KS, IQZ}),
which is the perfect valued derived category of the bigraded Ginzburg algebra constructed from $\surfo$.
The $\XX$-spherical objects in $\D_{\XX}(\surfo)$ correspond to the bigraded simple closed arcs in $\surfo$,
and their $\q$-dimensions of $\Hom$-spaces equal to the $\q$-intersections between the corresponding arcs \cite{KS, IQZ}.
One can specialize $\XX=N$, and $\D_{\XX}(\surfo)$ becomes a Calabi-Yau-$N$ category, for any integer $N\geq 2$.
When $N=3$, $\D_{3}(\surfo)$ provides an additive categeorification of cluster algebras of surface type
(cf. e.g. \cite{QQ}).
When $\surfo=\disk$ and $N=2$, we recover \cite{BBL}'s result (with a slight improvement).

The paper is organized as follows.
In section 2,
we recall several equivalent definitions of left and right $\q$-deformed rationals from \cite{MGO20, BBL},
via continued fractions, braid twist action and $\q$-weighted Farey graph respectively.
In section 3, we recall the graded decorated surface in the sense of \cite{KS, IQZ}
and prove the main results.
In section 4, we give the categorification and in section 5, we discuss reduction and the relation with Jones polynomials.

\subsection*{Acknowledgment}
Fl is grateful to Qy for leading her into this research area and providing her lots of help and supervisions. This work is inspired by the work of Morier-Genoud, Ovsienko \cite{MGO20} and Bapat, Becker, Licata \cite{BBL}.
Qy is supported by National Key R\&D Program of China (No.2020YFA0713000) and National Natural
Science Foundation of China (No.12031007 and No.12271279).

\section{\txq-deformed rationals and Farey graph}
In the paper, we fix the following conventions.
\begin{con}
\begin{itemize}
  \item Let \txq~ be a formal parameter.
  \item A rational number always belongs to $\uQ:=\QQ\cup \{\infty\}$. We also denote $\upQ:=\QQ^{\geq 0}\cup \{\infty\}$. We usually state the results for $\uQ$ but prove the non-negative case since the negative case holds by symmetry.
  \item We denote a rational number by $\rs$, including the exceptional cases when $0=\frac{0}{1}$ and $\infi=\frac{1}{0}$. We assume that $\rs$ is irreducible.
\end{itemize}
\end{con}
\subsection{Right and left \txq-deformed rationals}
We first recall the definitions of right and left \txq-deformations of rational numbers via finite continued fractions and formulate their basic properties.
For a positive rational $\rs$, it can be expressed as an expansion of continued fraction as
\begin{equation}\label{expansion}
\rat\quad=\quad a_1 + \cfrac{1}{a_2  + \cfrac{1}{\ddots +\cfrac{1}{a_{2m}} }}
\colon=[a _1,\ldots,a_{2m}],
\end{equation}
for $a_1\in\NN$ and $a_2,\cdots,a_{2m}\in\NN\setminus \{0\}$,
which is known as the ({\it regular}) continued fraction (expression). For the exceptional cases, we denote $0=[-1,1]$ and $\infi=[]$.

For a non-negative integer $a$, the \emph{right \txq-deformation} is defined as
$$[a]_{\q}^{\sharp}:=\frac{1-\qd^a}{1-\qd}=1+\qd+\qd^{2}+\cdots+\qd^{a-1},$$
and the corresponding \emph{left \txq-deformation} is defined as \[[a]_{\q}^{\flat}:=\frac{1-\qd^{a-1}+\qd^{a}-\qd^{a+1}}{1-\qd}=1+\qd+\cdots+\qd^{a-2}+\qd^a.\]
\begin{definition}[{\cite{MGO20,BBL}}]\label{qdef}
Let $\rat\in\ugQ$ be a rational number with continued fraction expansion $[a_{1}, \ldots, a_{2m}]$.
\begin{enumerate}
\item We define its \emph{right \txq-deformation} by the following formula:
\begin{equation}\label{qr}
\Big[\rat\Big]^{\sharp}_{\q} :=
      [a_{1}]^{\sharp}_{\q}+\cfrac{\qd^{a_{1}}}{[a_{2}]^{\sharp}_{\q^{-1}}+\cfrac{\qd^{-a_{2}}}{[a_{3}]^{\sharp}_{\q}+\cfrac{\qd^{a_{3}}}{[a_{4}]^{\sharp}_{\q^{-1}}+\cfrac{\qd^{-a_{4}}}{\cfrac{\ddots}{[a_{2m-1}]^{\sharp}_{\q}+\cfrac{\qd^{a_{2m-1}}}{[a_{2m}]^{\sharp}_{\q^{-1}}}}}}}}\quad.
\end{equation}
 \item We define its \emph{left $\q$-deformation} by the following formula:
 \begin{equation}\label{ql}
\Big[\rat\Big]_{\q}^{\flat}:=[a_{1}]^{\sharp}_{\q}+\cfrac{\qd^{a_{1}}}{[a_{2}]^{\sharp}_{\q^{-1}}+\cfrac{\qd^{-a_{2}}}{[a_{3}]^{\sharp}_{\q}+\cfrac{\qd^{a_{3}}}{[a_{4}]^{\sharp}_{\q^{-1}}+\cfrac{\qd^{-a_{4}}}{\cfrac{\ddots}{[a_{2m-1}]^{\sharp}_{\q}+\cfrac{\qd^{a_{2m-1}}}{[a_{2m}]^{\flat}_{\q^{-1}}}}}}}}\quad.
\end{equation}
\end{enumerate}
We normalize them as
\[\left[\rat\right]_{\q}^{\sharp}=\frac{\qr^{\sharp}(r/s)}{\qs^{\sharp}(r/s)}, \quad \left[\rat\right]_{\q}^{\flat}=\frac{\qr^{\flat}(r/s)}{\qs^{\flat}(r/s)},\]
so that the denominators are polynomials of $\q$ with lowest non-zero constant term. For $0$ and $\infty$, we set
\begin{gather*}
    \qr^{\sharp}(\zero)=0,\quad \qs^{\sharp}(\zero)=1,\quad \qr^{\flat}(\zero)=\q-1,\quad\qs^{\flat}(\zero)=\q;\\
    \qr^{\sharp}(\infi)=1, \quad \qs^{\sharp}(\infi)=0,\quad\qr^{\flat}(\infi)=1,\quad\qs^{\flat}(\infi)=1-\q.
\end{gather*}
\end{definition}

Next, we consider the group
\begin{equation*}
\PSL=\left\{\begin{pmatrix}a&b\\c&d\end{pmatrix}|a,b,c,d\in\ZZ,ad-bc=1\right\},
\end{equation*}
which is generated by
\[\gi:=\begin{pmatrix}1&1\\0&1\end{pmatrix},\quad \gii:=\begin{pmatrix}1&0\\-1&1\end{pmatrix}.\]
It acts on rational numbers $\uQ$ by linear fractional transformation as
\begin{equation}\label{eq:fl}
\begin{pmatrix}a&b\\c&d\end{pmatrix}\cdot\bigg(\cfrac{r}{s}\bigg)=\frac{ar+bs}{cr+ds},
\end{equation}
where $\frac{r}{s}\in\uQ$ and $\begin{pmatrix}a&b\\c&d\end{pmatrix}\in\PSL$.
For a rational number $\frac{r}{s}\in\upQ$ with continued fraction expansion as \eqref{expansion},
it is well-known that (cf. \cite[Proposition 2.2]{BBL})
\begin{equation}\label{eq:expre}
\cfrac{r}{s}=\gi^{a_1}\gii^{-a_2}\gi^{a_3}\gii^{-a_4}\cdots \gi^{a_{2m-1}} \gii^{-a_{2m}}(\cfrac{1}{0}).
\end{equation}

\begin{propdef}[{\cite{MGO20,BBL}}]\label{def2}
Consider the $\q$-deformation $\PSLq$ of $\mathrm{PSL}_{2}(\ZZ)$, which is generated by
\[\displaystyle
\gqi=\begin{pmatrix}\q&1\\0&1\\\end{pmatrix},\quad\gqii=\begin{pmatrix}1&0\\-\q&\q\\\end{pmatrix}.\]
For a rational number $\frac{r}{s}\in\upQ$ with expression \eqref{expansion}, we have
\begin{equation}
\begin{cases}
\left[\displaystyle\frac{r}{s}\right]^{\sharp}_\q&=\gqi^{a_1}\gqii^{-a_2}\gqi^{a_3}\gqii^{-a_4}
        \cdots \gqi^{a_{2m-1}}\gqii^{-a_{2m}}\left(\displaystyle\frac{1}{0}\right),\\
\left[\displaystyle\frac{r}{s}\right]_\q^{\flat}&=\gqi^{a_1}\gqii^{-a_2}\gqi^{a_3}\gqii^{-a_4}
        \cdots \gqi^{a_{2m-1}} \gqii^{-a_{2m}}\left(\displaystyle\frac{1}{1-\q}\right).
\end{cases}
\end{equation}
\end{propdef}
\subsection{\txq-deformations via Farey graph}
The classical {\it Farey graph} $\FG$ is an infinite graph with vertices set
$$\FGv=\uQ.$$
There is an edge between $\frac{p}{q}$ and $\frac{u}{v}$ if and only if $pv-uq=\pm1$ (see Figure \ref{fig:FG}).
If $\frac{p}{q}$ and~$\frac{u}{v}$ are connected by an edge, we define their \emph{Farey sum} by
$$\frac{p}{q}\oplus\frac{u}{v}:=\frac{p+u}{q+v}.$$
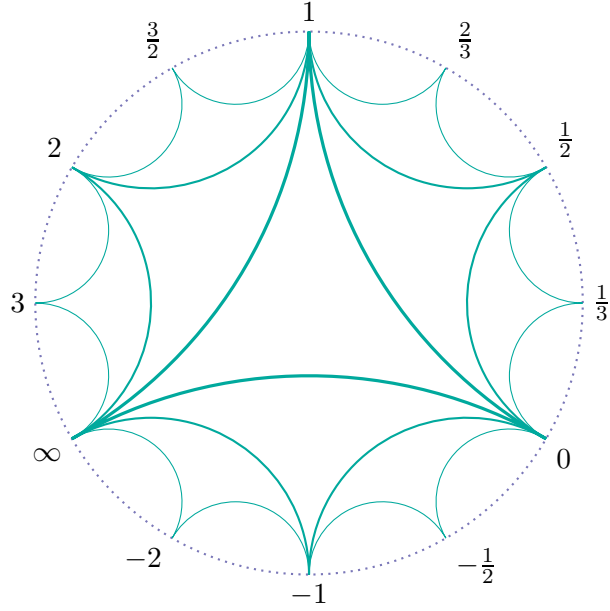
\begin{figure}\centering
\begin{tikzpicture}[scale=0.6,
 arrow/.style={->,>=stealth,thick,blue}, 
 border/.style={Periwinkle,dotted,thick},
 c-vrtx/.style={blue},
 c-arc/.style={Emerald}]
\newcommand{\vrtx}{\bullet}
\coordinate (O) at (0,0);
\coordinate (S1) at (0,6) ;
\draw [border] (O) circle (6cm);
\draw (0:6cm) node[right] {$\frac{1}{3}$} (-30:6cm) node[below right] {$0$} (30:6cm) node[above right] {$\frac{1}{2}$} (60:6cm) node[above right] {$\frac{2}{3}$}
    (90:6cm) node[above] {$1$} (120:6cm) node[above left] {$\frac{3}{2}$}
    (-90:6cm) node[below] {$-1$} (-60:6cm) node[below right] {$-\frac{1}{2}$} (-120:6cm) node[below left] {$-2$}
    (210:6cm) node[below left] {$\infty$} (180:6cm) node[left] {$3$} (-210:6cm) node[above left] {$2$};
\draw[c-arc,very thick] (S1) \foreach \j in {1,...,3}
    {arc(360/3-\j*360/3+180:360-\j*360/3:10.3923cm)}--cycle;
\draw[c-arc,thick] (S1) \foreach \j in {1,...,6}
    {arc(360/6-\j*360/6+180:360-\j*360/6:3.4641cm)}--cycle;
\draw[c-arc] (S1) \foreach \j in {1,...,12}
    {arc(360/12-\j*360/12+180:360-\j*360/12:1.6077cm)}--cycle;
\end{tikzpicture}
\caption{The Farey graph}
\label{fig:FG}
\end{figure}

Moreover, $\FGv$ is parametrized by homotopy classes of simple closed arcs on torus with at most one boundary and the edges are those arcs with intersection number one. $\PSL$ acts on $\FG$ by \eqref{eq:fl} taking one edge to another. In particular, if $\Ti\in\PSL$ takes the form
\[\Ti_{\frac{r}{s}}=\begin{pmatrix}1+rs&-r^2\\s^2&1-rs\end{pmatrix},\]
then it is a rotation which fixes $\frac{r}{s}$.
\begin{lemdef} [{\cite[Section 2.2]{MGO20}}]\label{rmk:deco}
Let $\frac{r}{s}\in\ugQ$ be any rational number with continued fraction expansion as \eqref{expansion}. Then it can be uniquely written as Farey sum of two rationals $\frac{p}{q},\frac{u}{v}\in\upQ$, i.e.
\[\rat=\frac{p}{q}\oplus\frac{u}{v},\]
with $uq-pv=1$ and $\frac{p}{q}<\frac{r}{s}<\frac{u}{v}$. In fact,
\begin{equation}
\frac{p}{q}=
\begin{cases}
[a_1,a_2,\ldots,a_{2m-2}+1],&~\mbox{if}~a_{2m-1}=1\mbox{and}~m>1;\\
[a_1,a_2,\ldots,a_{2m-1}-1,1],&~\mbox{otherwise},
\end{cases}
\end{equation}
and
\begin{equation}
\frac{u}{v}=
\begin{cases}
[a_1,a_2,\ldots,a_{2m-1},a_{2m}-1],&~\mbox{if}~a_{2m}\geq 2;\\
[a_1,a_2,\ldots,a_{2m-2}],&~\mbox{if}~a_{2m}=1.
\end{cases}
\end{equation}
Moreover, there is an associated integer defined as
\begin{equation}
l=l(\frac{r}{s})=
\begin{cases}
0,&~\mbox{if}~a_{2m}\geq 2;\\
a_{2m-1},&~\mbox{if}~a_{2m}=1.
\end{cases}
\end{equation}
In particular, we have $l(n+1=\frac{n}{1}\oplus\frac{1}{0})=n-1$.
\end{lemdef}On the other hand, $l$ can also be defined for an edge in $\FG$ connecting $\frac{p}{q}$ and $\frac{u}{v}$, provided $\frac{p}{q}<\frac{u}{v}$. More precisely, $l(\frac{p}{q},\frac{u}{v}):=l(\frac{p}{q}\oplus\frac{u}{v})$.

As in \cite{MGO20}, we assign a weight to each edge of the Farey graph, that goes along with the right or left $\q$-deformations associated to vertices. Then the right and left \txq-deformations can also be defined via $\q$-Farey sum.

\begin{propdef}[{\cite{MGO20}}]\label{qfareydef}
Let $\frac{r}{s}\in\ugQ$ be a rational with the decomposition $\rat=\frac{p}{q}\oplus\frac{u}{v}$ and $l=l(\frac{p}{q},\frac{u}{v})$ as above. For right \txq-deformation, we have
\emph{
\begin{equation}
\qr^{\sharp}(\frac{r}{s}):=\qr^{\sharp}(\frac{p}{q})+\qd^{l+1}\qr^{\sharp}(\frac{u}{v}),\quad\qs^{\sharp}(\frac{r}{s}):=\qs^{\sharp}(\frac{p}{q})+\qd^{l+1}\qs^{\sharp}(\frac{u}{v}).
\end{equation}}
For left \txq-deformation, if we set \emph{$\qrb{\zero}=\qs^{\flat}(\zero), \qrb{\infi}=\qs^{\flat}(\infi)$} and define
\emph{
\begin{equation}
\qrb{\rs}:=\qrb{\frac{p}{q}}+\qid^{l+1}\qrb{\frac{u}{v}},\quad\qsb{\rs}:=\qsb{\frac{p}{q}}+\qid^{l+1}\qsb{\frac{u}{v}},
\end{equation}}
then we have
\emph{
\[ \left[\rat\right]_{\q}^{\flat}=\frac{\qrb{r/s}}{\qsb{r/s}}.\]}
\end{propdef}
We label a weight to each edge and a right or left \txq-deformation to each vertex in Farey graph,  which are drawn in Figure \ref{rqfar} and Figure \ref{lqfar} respectively. Here the integer $l=l(\frac{p}{q},\frac{u}{v})$ as above.
\begin{figure}[ht]
\begin{center}
\begin{tikzpicture}[scale=5]
    \draw[Emerald,very thick] (1,0) arc (0:180:.5);
    \draw  (0,0)  node[below]{$\displaystyle\left[\frac{0}{1}\right]_\q^{\sharp}$};
    \draw  (1,0)  node[below]{$\displaystyle\left[\frac{1}{0}\right]_\q^{\sharp}$};
    \draw[Emerald,thick] (1,0) arc (0:180:.25);
    \draw  (.5,0) node[below]{$\displaystyle\left[\frac{1}{1}\right]_\q^{\sharp}$};
    \draw[Emerald,thick] (.5,0) arc (0:180:.25);
\draw (.25,.25) node[above]{$1$};
\draw (.75,.25) node[above]{$1$};
\draw (.5,.5) node[above]{$\qd^{-1}$};
\begin{scope}[xshift=40]
\draw[Emerald,very thick] (1,0) arc (0:180:.5);
    \draw  (0,0)  node[below]{$\displaystyle\frac{\qr^{\sharp}(\pq)}{\qs^{\sharp}(\pq)}$};
    \draw  (1,0)  node[below]{$\displaystyle\frac{\qr^{\sharp}(\rsi)}{\qs^{\sharp}(\rsi)}$};
    \draw[Emerald,thick] (1,0) arc (0:180:.25);
    \draw  (.5,0) node[below]{$\displaystyle\frac{\qr^{\sharp}(r/s)}{\qs^{\sharp}(r/s)}$};
    \draw[Emerald,thick] (.5,0) arc (0:180:.25);
\draw (.25,.25) node[above]{$1$};
\draw (.75,.25) node[above]{$\qd^{l+1}$};
\draw (.5,.5) node[above]{$\qd^{l}$};
\end{scope}
\end{tikzpicture}
\end{center}
\caption{The right \txq-deformation via \txq-Farey sum}
    \label{rqfar}
\end{figure}
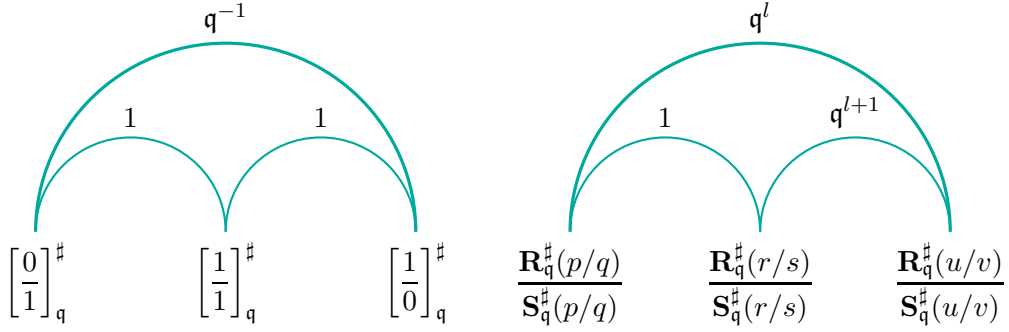

\begin{figure}[ht]\centering
\begin{tikzpicture}[scale=5]
    \draw[Emerald,very thick] (1,0) arc (0:180:.5);
    \draw  (0,0)  node[below]{$\displaystyle\left[\frac{0}{1}\right]_\q^{\flat}$};
    \draw  (1,0)  node[below]{$\displaystyle\left[\frac{1}{0}\right]_\q^{\flat}$};
    \draw[Emerald,thick] (1,0) arc (0:180:.25);
    \draw  (.5,0) node[below]{$\displaystyle\left[\frac{1}{1}\right]_\q^{\flat}$};
    \draw[Emerald,thick] (.5,0) arc (0:180:.25);
\draw (.25,.25) node[above]{$1$};
\draw (.75,.25) node[above]{$1$};
\draw (.5,.5) node[above]{$q$};
\begin{scope}[xshift=40]
\draw[Emerald,very thick] (1,0) arc (0:180:.5);
    \draw  (0,0)  node[below]{$\displaystyle\frac{\qrb{\pq}}{\qsb{\pq}}$};
    \draw  (1,0)  node[below]{$\displaystyle\frac{\qrb{\rsi}}{\qsb{\rsi}}$};
    \draw[Emerald,thick] (1,0) arc (0:180:.25);
    \draw  (.5,0) node[below]{$\displaystyle\frac{\qrb{r/s}}{\qsb{r/s}}$};
    \draw[Emerald,thick] (.5,0) arc (0:180:.25);
\draw (.25,.25) node[above]{$1$};
\draw (.75,.25) node[above]{$\qid^{l+1}$};
\draw (.5,.5) node[above]{$\qid^{l}$};
\end{scope}
\end{tikzpicture}
\caption{The left \txq-deformation via \txq-Farey sum}
    \label{lqfar}
\end{figure}
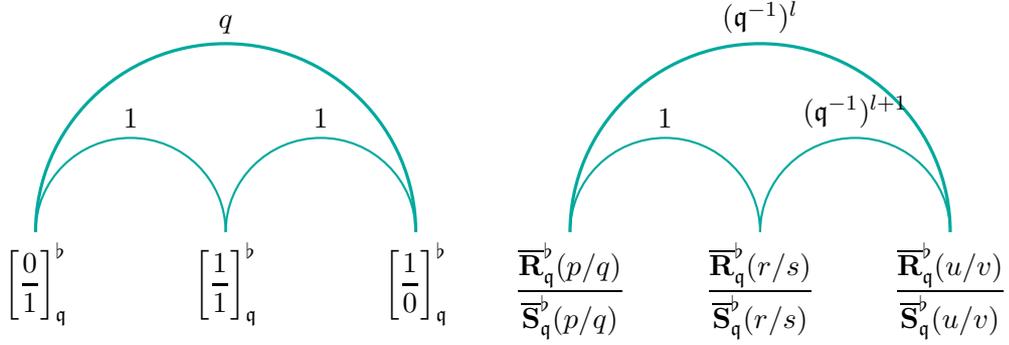
\begin{remark}\label{def:neg}
For $\frac{r}{s}\in\QQ^+$ with continued fraction expression $[a_1,\ldots, a_{2m}]$, we have
\begin{equation}
-\cfrac{r}{s}=[-a_1,\ldots,-a_{2m}],
\end{equation} and
\begin{equation}
-\cfrac{r}{s}=\gi^{-a_1}\gii^{a_2}\gi^{-a_3}\gii^{a_4}\cdots \gi^{-a_{2m-1}} \gii^{a_{2m}}(\cfrac{1}{0}).
\end{equation}
The right and left $\q$-deformations for negative rational numbers are defined as:
\begin{equation}
\begin{cases}
\left[\displaystyle-\frac{r}{s}\right]^{\sharp}_\q&=\gqi^{-a_1}\gqii^{a_2}\gqi^{-a_3}\gqii^{a_4}
        \cdots \gqi^{-a_{2m-1}}\gqii^{a_{2m}}\left(\displaystyle\frac{1}{0}\right),\\
\left[\displaystyle-\frac{r}{s}\right]_\q^{\flat}&=\gqi^{-a_1}\gqii^{a_2}\gqi^{-a_3}\gqii^{a_4}
        \cdots \gqi^{-a_{2m-1}} \gqii^{a_{2m}}\left(\displaystyle\frac{1}{1-\q}\right).
\end{cases}
\end{equation}
In fact, we can obtain $\q$-deformations of negative rationals from positive ones by the following formula:
\begin{equation}
\left[-\rat\right]_{\q}^*:=-\q^{-1}\left[\frac{r}{s}\right]_{\q^{-1}}^*,
\end{equation}
where $*\in\{\sharp, \flat\}$.
\end{remark}

\section{The topological model}
In this section, we introduce decorated (marked) surfaces as the topological model which we will use.
We first summarize the setting and results in \cite{KS,IQZ} and then
show that the \txq-intersections of certain arcs describe the left/right \txq-deformations for rational numbers.
\subsection{Decorated surfaces}
Let $\surf$ be an oriented surface with non-empty boundary $\bdy$ and we denote its interior by $\isurf=\surf\setminus(\bdy)$.
We decorate $\surf$ with a finite set $\Tri$ of points (\emph{decorations}) in $\isurf$,
denoted by $\surfo$.

Let $\surfoi=\surf\setminus(\partial\surf\cup\Tri)$. An \emph{arc} $c$ in $\surfo$ is a curve $c:[0,1]\to \surf$ such that $c(t)\in\surfoi$ for any $t\in(0,1)$. The \emph{inverse} $\overline{c}$ of an arc $c$ is defined as $\overline{c}(t)=c(1-t)$ for any $t\in[0,1]$.
\begin{definition}\label{not:CA}
A \emph{closed arc} $c$ is an arc whose endpoints $c(0)$ and $c(1)$ are in $\Tri$. It is \emph{simple} if moreover it satisfies $c(0)\neq c(1)$, without self-intersections in $\surfoi$. We denote by $\CA(\surfo)$ the set of simple closed arcs.
\end{definition}
In this paper, we always consider arcs up to taking inverse and homotopy relative to endpoints and exclude the arcs which are isotopic to a point in $\surfo$. Two arcs are in \emph{minimal position} if their intersection is minimal in the homotopy class. For three simple closed arcs, they form a \emph{contractible triangle} if they bound a disk which is contractible. For $\us, \ut\in\CA(\surfo),$ we use the notations
\begin{equation*}
\Int_{\Pi}(\us,\ut):=\left|\cap(\us,\ut)\cap \Pi\right|,
\end{equation*}
where $\Pi=\Tri$ or $\surfoi$, a subset of $\surfo$, and
\begin{equation}
\Int_{\surfo}(\us,\ut):= \frac{1}{2}\cdot\Int_{\Tri}(\us,\ut)+ \Int_{\surfoi}(\us,\ut).
\end{equation}
The \emph{mapping class group} $\MCG(\surfo)$ of a decorated surface $\surfo$ consists of the isotopy classes of the homeomorphisms of $\surf$ which fix $\partial\surf$ pointwise and fix $\triangle$ setwise. For any $\alpha\in\CA(\surfo)$, the associated \emph{braid twist} $B_\alpha\in\MCG(\surfo)$ is defined in Figure~\ref{fig:bt}. We have the formula
\begin{equation}\label{eq:bt}
B_{\Psi(\alpha)}=\Psi\circ B_\alpha\circ\Psi^{-1}
\end{equation}
for any $\alpha\in\CA(\surfo)$ and $\Psi\in\MCG(\surfo)$. We define $\BT(\surfo)$ to be the subgroup of $\MCG(\surfo)$ generated by $B_{\alpha}$ for $\alpha\in\CA(\surfo)$.
\begin{figure}[ht]\centering
	\begin{tikzpicture}[scale=.3]
	\draw[very thick,dashed](0,0)circle(6);
	\draw(-2,0)edge[red, very thick](2,0)  edge[cyan,very thick, dashed](-6,0);
	\draw(2,0)edge[cyan,very thick,dashed](6,0);
	\draw(-2,0)node[white] {$\bullet$} node[red] {$\circ$};
	\draw(2,0)node[white] {$\bullet$} node[red] {$\circ$};
	\draw(0:7.5)edge[very thick,->,>=latex](0:11);\draw(0:9)node[above]{$B_{\alpha}$};
	\end{tikzpicture}\;
	\begin{tikzpicture}[scale=.3,yscale=-1]
	\draw[very thick, dashed](0,0)circle(6)node[above,black]{$_\alpha$};
	\draw[red, very thick](-2,0)to(2,0);
	\draw[cyan,very thick, dashed](2,0).. controls +(0:2) and +(0:2) ..(0,-2.5)
	.. controls +(180:1.5) and +(0:1.5) ..(-6,0);
	\draw[cyan,very thick,dashed](-2,0).. controls +(180:2) and +(180:2) ..(0,2.5)
	.. controls +(0:1.5) and +(180:1.5) ..(6,0);
	\draw(-2,0)node[white] {$\bullet$} node[red] {$\circ$};
	\draw(2,0)node[white] {$\bullet$} node[red] {$\circ$};
	\end{tikzpicture}
	\caption{The braid twist}
	\label{fig:bt}
\end{figure}
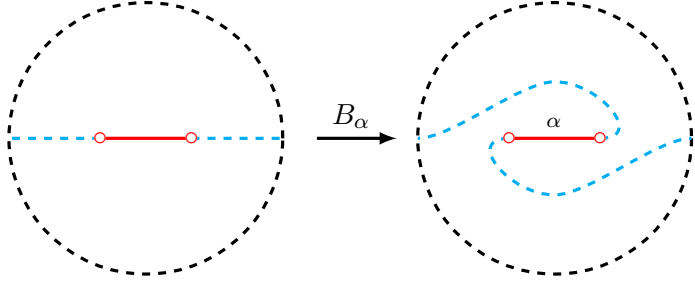
The braid twist can be illustrated by smoothing out.
\begin{construction}
For any $\sigma,\tau\in\CA(\surfo)$ with $\sigma(0)=\tau(0)=z\in\Tri$. The extension $\sigma\wedge\tau$ of $\sigma$ by $\tau$ (with respect to the common starting point) is defined in Figure~\ref{fig:ext.},
which is the operation of smoothing out the intersection moving from $\sigma$ to $\tau$ clockwisely.
	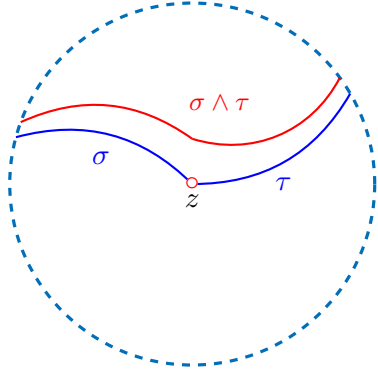
\begin{figure}[ht]\centering
		\begin{tikzpicture}[scale=1.2]
		\draw[NavyBlue,dashed,very thick](0,0)circle(2);
		\draw[thick,blue](-195:2)edge[bend left,>=stealth](0,0)
		(0,0)edge[bend right,>=stealth](30:2)
		(1,0)node{$\tau$}(-1,.1)node[above]{$\sigma$};
		\draw[thick,red](160:2)to[bend left](0,.5)(0.3,.7)
		node[above]{\small{$\sigma\wedge\tau$}};
		\draw[thick,red](0,0.5)edge[bend right=40,>=stealth](36:2);
		\draw(0,0)node[white] {$\bullet$} node[red](a){$\circ$} (0,0)node[below]{$z$};
		\end{tikzpicture}
		\caption{The extension as smoothing out}
		\label{fig:ext.}
	\end{figure}
\end{construction}
Notice that if $\Int_{\surfo}(\us,\ut)=\frac{1}{2}$, i.e. they only intersect at one decoration, then
\begin{equation}\label{eq:ext}
\sigma\wedge\tau=B_{\tau}(\sigma)=B^{-1}_{\sigma}(\tau).
\end{equation}
\begin{lemma}\label{lem:CA}
Assume that $|\Tri|\geq 3$. For any $\eta\in\CA(\surfo)$, we have
\begin{equation}
\CA(\surfo)=\BT(\surfo)\cdot\{\eta\}.
\end{equation}
\end{lemma}
\begin{proof}
Let $\xi$ be a simple closed arc in $\CA(\surfo)$. We notice that the intersection $\Int_{\surfo}(\ue,\xi)\in\cfrac{1}{2}\cdot\ZZ^+$. When $\Int_{\surfo}(\ue,\xi)>0$, we use induction on it. For the starting case when $\Int_{\surfo}(\ue,\xi)=\frac{1}{2}$, we take $\ua=\xi\wedge\ue\in\CA(\surfo)$, and then $\xi=B_{\ua}(\ue)$ by \eqref{eq:ext}.

Now suppose the assertion holds when $\Int_{\surfo}(\ue,\xi)<k$ and consider the case when $\Int_{\surfo}(\ue,\xi)=k\in\cfrac{1}{2}\cdot\ZZ^+$. There exists some decoration $z$ which is not the endpoint of $\xi$ (if the endpoints of $\ue$ and $\xi$ are not coincide, we take $z$ to be an endpoint of $\ue$) and we connect $z$ to $\xi$ by $l$ such that it intersects $\xi$ at $p$ (see Figure \ref{fig:xi12}). We cut $\xi$ at $p$ and connect two parts with $l$ respectively. Then we obtain $\xi_1,\xi_2\in\CA(\surfo)$ such that $\xi=\xi_1\wedge\xi_2$ and
\begin{equation}
\Int_{\surfo}(\ue,\xi)=\Int_{\surfo}(\ue,\xi_1)+\Int_{\surfo}(\ue,\xi_2).
\end{equation}
By assumption, there exists $b\in\BT(\surfo)$ such that $\xi_1=b(\ue)$. Thus $\xi=B_{\xi_2}(\xi_1)= (B_{\xi_2}\cdot b)(\ue)$.
\begin{figure}[ht]\centering
\begin{tikzpicture}[yscale=-1]
\draw[red,thick](0,0).. controls +(90:1) and +(-90:1) ..(3,2);
\draw[red,thick](6,0).. controls +(180:1) and +(0:1) ..(3,2);

\draw[red,thick](0,0).. controls +(90:.7) and +(210:1) ..(3,1)
    node[black]{\tiny{$\bullet$}} node[black,above]{$p$}
    .. controls +(30:1) and +(180:1) ..(6,0);
\draw[bend right=15,thick](3,1)to(3,2) (3.1,1.5)node[right]{$l$};

\draw(0,0)node[white] {$\bullet$} node[red]{$\circ$}
     (3,2)node[white] {$\bullet$} node[red]{$\circ$}node[below]{$z$}
     (6,0)node[white] {$\bullet$} node[red]{$\circ$}
     (4.8,1)node {$\xi_2$}(1.5,1.3)node {$\xi_1$}(3.7,0.9)node[above] {$\xi$};
\end{tikzpicture}
\caption{Decomposing $\xi$}
\label{fig:xi12}
\end{figure}
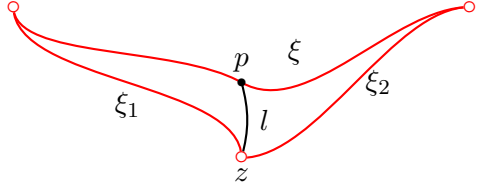

Finally, we consider the case when $\Int_{\surfo}(\ue,\xi)=0$. We can choose a simple closed arc $\ua$ such that $\Int_{\surfo}(\ue,\ua)=\Int_{\surfo}(\xi,\ua)=\frac{1}{2}$. By the starting case, we have $\ua\in\BT(\surfo)\cdot\ue$ and $\xi\in\BT(\surfo)\cdot\ua\subset\BT(\surfo)\cdot\ue$.
\end{proof}
\def\di{z_{\infi}}
\def\dii{z_*}
\def\diii{z_{\zero}}
\paragraph{\textbf{The branched double cover}}
Let $\bc$ be a branched double cover of $\surfo$ branching at decorations. When there exists extra structure (e.g. a quadratic differential), $\bc$ can be constructed as the spectral cover (cf. \cite{KS}). We denote the covering map by $\cpi:\bc\to\surfo$.

We consider the special case when $\surfo=\disk$ is a disk and $\Tri=\{\deci,\decs,\decz\}$. We fix two initial simple closed arcs $\ue_{\zero}$ and $\ue_{\infi}$ such that $\ue_{\zero}\cap\ue_{\infi}=\{\decs\}$, see Figure \ref{fig:disk}. Notice that there is an anti-clockwise angle from $\ue_{\infi}$ to $\ue_{\zero}$. The branched double cover $\bc$ is a torus with one boundary $\partial_\Sigma$. For simplicity, we draw $\partial_\Sigma$ as a puncture in the figures.
\begin{figure}[ht]\centering
\begin{tikzpicture}[scale=.4,yscale=.8]
	\draw[very thick](0,0)circle(6);
\draw[blue!50,-<-=.6,>=stealth,thick](180:.8)arc(180:0:.8)(0:.8);
\draw(-3,0)edge[red, very thick](0,0);
\draw(0,0)edge[red, very thick](3,0);
\draw(-3,0)node[below,black]{$\decz$} (0,0)node[below,black]{$\decs$} (3,0)node[below,black]{$\deci$};
\draw(-1.5,0)node[above,red]{$\ue_{\zero}$} (1.5,0)node[above,red]{$\ue_{\infi}$};
	\draw(-3,0)node[white] {$\bullet$} node[red] {$\circ$};
	\draw(0,0)node[white] {$\bullet$} node[red] {$\circ$};
	\draw(3,0)node[white] {$\bullet$} node[red] {$\circ$};
	\end{tikzpicture}
\caption{$\disk$ is a disk with three decorations}\label{fig:disk}
\end{figure}
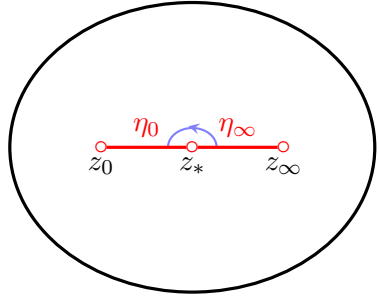
We take the $\ZZ^2$-covering $\ubc$ of $\bc$, where the white area is a fundamental domain (see Figure \ref{fig:univfundom}), and we denote the covering map by $\widetilde{\pi}:\ubc\to\bc$. When forgetting the punctures and decorations of $\ubc$, it is the universal cover of torus. Hence we embed $\ubc$ into $\RR^2$ such that all decorations and punctures are integer points, where its fundamental domain is a unit square. For each line (it is not allowed to pass through the punctures) in $\ubc$ with rational slope $\frac{r}{s}\in\uQ$, it becomes a simple closed curve $C_{\frac{r}{s}}$ in $\bc$ under the map $\widetilde\pi$.
\begin{figure}[ht]\centering
\begin{tikzpicture}[scale=.5]
\draw[white,fill=gray!20](0,0)rectangle(12,12);
\draw[white,fill=white](0,0)rectangle(6,6);
\draw[very thick](0,0)--(12,0)--(12,12)--(0,12)--(0,0) (0,6)--(12,6) (6,12)--(6,0);
\draw (3,0)node[below]{$\deci$} (0,3)node[left]{$\decz$}(3,12)node[above]{$\deci$}(12,3)node[right]{$\decz$}(3.5,3)node[below]{$\dii$}(9,0)node[below]{$\deci$} (0,9)node[left]{$\decz$}(9,12)node[above]{$\deci$}(12,9)node[right]{$\decz$};
\draw[red, very thick] (3,0)--(3,12) (0,3)--(12,3) (9,0)--(9,12)(0,9)--(12,9);
\draw[red, very thick](3,0)--(6,3)--(9,9) (3,0)--(9,9);
\draw (0,0)\nn  (6,0)\nn (0,6)\nn (6,6)\nn  (12,0)\nn  (6,12)\nn (12,6)\nn (12,12)\nn  (0,12)\nn ;
\draw[white] (3,0)\nn  (6,3)\nn (0,3)\nn (3,6)\nn  (3,3)\nn  (9,0)\nn  (9,3)\nn (12,3)\nn (9,6)\nn  (12,9)\nn  (0,9)\nn  (3,9)\nn (6,9)\nn (9,9)\nn  (3,12)\nn  (9,12)\nn ;
\draw[red](3,0)\ww (0,3)\ww(3,6)\ww(6,3)\ww(3,3)\ww (9,0)\ww(9,3)\ww(12,3)\ww(9,6)\ww(12,9)\ww(0,9)\ww(3,9)\ww(6,9)\ww(9,9)\ww(3,12)\ww(9,12)\ww;
	\end{tikzpicture}
\caption{The $\ZZ^2$-covering of $\bc$ and its fundamental domain}
\label{fig:univfundom}
\end{figure}

\begin{lemma}[{\cite[\S 10]{KQ}}]\label{cor:biCA}
The set $\CA(\disk)$ of simple closed arcs in $\disk$ can be parameterized by rational numbers, i.e there is a bijection
\begin{equation}
\begin{tikzpicture}[xscale=.6,yscale=.6]
\draw(180:3)node(o){$\CA(\disk)$}(-3,2.2)node(b){\small{$\Br_3/\Z(\Br_3)$}}
(0,2.5)node{$\iota$}(0,.5)node{$\cong$};
\draw(0:3)node(a){$\uQ$}(3,2.2)node(s){\small{$\PSL$}};
\draw[->,>=stealth](o)to(a);\draw[->,>=stealth](b)to(s);
\draw[->,>=stealth](-3.2,.6).. controls +(135:2) and +(45:2) ..(-3+.2,.6);
\draw[->,>=stealth](3-.2,.6).. controls +(135:2) and +(45:2) ..(3+.2,.6);
\end{tikzpicture}
\end{equation}
sending $\ue_{\frac{r}{s}}$ to $\frac{r}{s}$.
\end{lemma}
\begin{proof}
We lift simple closed arcs in $\CA(\surfo)$ to simple closed curves in $\bc$, which can be parametrized by rational numbers in $\uQ$. That is, for any $\frac{r}{s}\in\uQ$, there exists an $\ue_{\frac{r}{s}}\in\CA(\surfo)$ which corresponds to a simple closed curve $C_{\frac{r}{s}}$ in $\bc$. Notice that the homology group $H_1(\bc)=\ZZ[C_{\infi}]\oplus\ZZ[C_0]$. Thus the homology class $[C_{\frac{r}{s}}]$ corresponds to $(r,s)$ in $H_1(\bc)\cong\ZZ^2$. Notice that the braid twist $\BT(\disk)\cong\Br_3$ lifts to Dehn twist $\DT(\bc)\subset\MCG(\bc)$, which is generated by $C_0$ and $C_\infi$. By the identification $\DT(\bc)/\Z(\DT(\bc))\cong\PSL$, the lemma holds.
\end{proof}
\begin{remark}
Here is a consequence of the lemma above. For a rational number $\frac{r}{s}\in\QQ^+$ with expression \eqref{eq:expre}, the corresponding arc in $\CA(\disk)$ is
\begin{equation}
\ue_{\frac{r}{s}}=B_{\ue_{\infi}}^{a_1}B_{\ue_{\zero}}^{-a_2}B_{\ue_{\infi}}^{a_3}B_{\ue_{\zero}}^{-a_4}\cdots B_{\ue_{\infi}}^{a_{2m-1}} B_{\ue_{\zero}}^{-a_{2m}}(\ue_{\infi}).
\end{equation}
Moreover,
\begin{equation}
\ue_{-\frac{r}{s}}=B_{\ue_{\infi}}^{-a_1}B_{\ue_{\zero}}^{a_2}B_{\ue_{\infi}}^{-a_3}B_{\ue_{\zero}}^{a_4}\cdots B_{\ue_{\infi}}^{-a_{2m-1}} B_{\ue_{\zero}}^{a_{2m}}(\ue_{\infi}).
\end{equation}
\end{remark}
\subsection{Bigraded arcs and \txq-intersections}\label{sec:dgc}
Let $\surfo$ be a decorated surface. In this section, we define the bigraded arcs and their \txq-intersections. Let $\PTSt=\PTSO$ be the real projectivization of the tangent bundle of $\surf\setminus\Tri$. We want to introduce a particular covering of $\PTSt$ with covering group $\ZZ\oplus\ZZ\XX\cong\ZZ^2$. A \emph{grading} $\Lambda:\surfo\to\PTSt$ on $\surfo$ is determined by a class in $\ho{1}(\PTSt,\ZZ\oplus\ZZ\XX)$, with value $1$ on each anti-clockwise loop $\{p\}\times\RR\mathbb{P}^1$ on $\PTSo{p}$ for $p\notin\Tri$ and value $-2+\XX$ on each anti-clockwise loop $l_z\times\{x\}$ on $\surfo$ around any $z\in\Tri,x\in\RR\mathbb{P}^1$. For any simple loop $\alpha$ on $\surfo$, we denote $\Li(\alpha)$ the $\ZZ$ part of $\Lambda(\alpha)$ and denote $\Lambda_2(\alpha)$ the $\ZZ\XX$ part of $\Lambda(\alpha)$. In fact, the first grading is a line field $\lambda$ of $\surfo$ which is determined by a class in $\ho{1}(\PTSt,\ZZ)$. Define $\wPTSt$ to be the $\ZZ\oplus\ZZ\XX$ covering of $\PTSt$ classified by the grading $\Lambda$, and denote the $\ZZ\oplus\ZZ\XX$ action on $\wPTSt$ by $\chi$.
\begin{definition}\label{def:bgDMS}
A \emph{graded decorated surface} $\DMS$ consists of a decorated surface $\surfo$ and a grading $\Lambda$ on $\surfo$.
\end{definition}
Let $\DMS$ be a graded decorated surface and $c:[0,1]\to \surf$ be an arc in $\surf$. There is a canonical section $s_c: c\setminus\Tri\to\PTSt$ given by $s_c(z)=T_zc$. A \emph{bigrading} on $c$ is given by a lift $\lc$ of $s_c$ to $\wPTSt$. The pair $(c,\lc)$ is called a \emph{bigraded arc}, and we usually denote it by $\lc$. Note that there are $\ZZ^2$ lifts of $c$, which are related by the $\ZZ^2$-action defined by $\chi$. One is the shift grading such that $\lc[m](t)=\chi(m,0)\lc(t)$ and the other is the $\XX$-grading such that $\lc\{m\}(t)=\chi(0,m)\lc(t)$ for any $m\in\ZZ$. For any $\ue\in\CA(\surfo)$, we call any of its lifts $\lwe$ in $\wPTSt$ a \emph{bigraded simple closed arc}. Denote by $\wXCA(\surfo)$ the set of bigraded simple closed arcs.

For any bigraded arcs $\lws$ and $\lwt$ which are in minimal position with respect to each other, let $p=\us(t_1)=\ut(t_2)\in\isurf$ be the point where $\us$ and $\ut$ intersect transversally.  Fix a small circle $a\subset\surf\setminus\Tri$ around $p$. Let $\alpha:[0,1]\to a$ be an embedded arc which moves anti-clockwise around $p$, such that $\alpha$ intersects $\us$ and $\ut$ at $\alpha(0)$ and $\alpha(1)$, respectively (cf. Figure~\ref{fig:iid}). If $p\in\Tri$, then $\alpha$ is unique up to a change of parametrization (see the left one of Figure~\ref{fig:iid}); otherwise there are two possibilities, which are distinguished by their endpoints (see the right one of Figure~\ref{fig:iid}). Take a smooth path $\rho:[0,1]\to\PTSt$ with $\rho(t)\in\PP T_{\alpha(t)}\surfo$ for all $t$, going from $\rho(0)=T_{\alpha(0)}\us$ to $\rho(1)=T_{\alpha(1)}\ut$, such that $\rho(t)\neq T_{\alpha(t)}l$ for all $t$. Lift $\rho$ to a path $\widehat{\rho}:[0,1]\to\wPTSt$ with $\widehat{\rho}(0)=\lws(\alpha(0))$. Then there exist some integers $\ii,\sii\in\ZZ$ such that $\lwt(\alpha(1))=\chi(\ii+\sii\XX)\widehat{\rho}(1)$.
\begin{figure}[htpb]
\begin{tikzpicture}[scale=.6]
\draw[blue!50,thick,->-=.5,>=stealth](45:1)arc(45:135:1) (135:1);
\draw[red] (90:1)node[above]{$\ua$};
\draw[red,thick](0,0)to node[below]{$\quad\us$}(45:4.5);
\draw[red,thick](0,0)to node[below]{$\ut\quad$}(135:4.5);
\draw[red](0,0)node[white]{$\bullet$} \ww;
\draw[red](0,0)node[below]{$p$};
\begin{scope}[shift={(8,0)}]
	\draw[blue!50,thick,->-=.5,>=stealth](45:1)arc(45:135:1) (135:1);
	\draw[blue!50,thick,->-=.5,>=stealth](225:1)arc(225:315:1) (315:1);
	\draw[red] (90:1)node[above]{$\ua$};
		\draw[red] (270:1)node[below]{$\ua$};
	\draw[red,thick](0,0)to node[below]{$\quad\us$}(45:4);
	\draw[red,thick](0,0)to (180+45:4);
	\draw[red,thick](0,0)to node[below]{$\ut\quad$}(135:4);
	\draw[red,thick](0,0)to (315:4);
 \draw[red](0,0)node[red]{$\bullet$};
	\draw[red](0,0)node[below]{$p$};
	\end{scope}
	\end{tikzpicture}
	\caption{Intersection at $p$ when $p$ is a decoration or not}\label{fig:iid}
\end{figure}
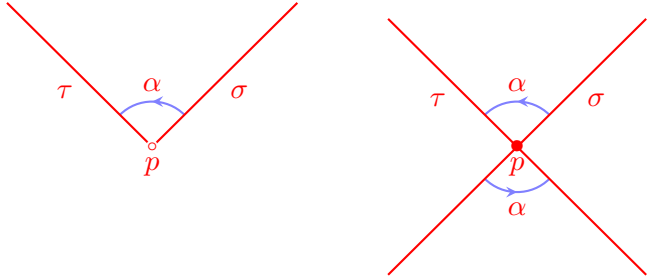
\begin{definition}[{\cite{KS}}]\label{def:bi-index}
For any bigraded arcs $\lws$ and $\lwt$ in $\DMS$, we call an intersection $p$ of $\lws$ and $\lwt$ with \emph{bi-index}
\begin{equation}\label{eq:Z2-int}
\ind_p^{\ZZ^2}(\lws,\lwt)=\ind_p(\lws,\lwt)+\ind_p^{\XX}(\lws,\lwt)\XX,
\end{equation}
where $\ind_p(\lws,\lwt):=\ii$ and $\sind_p(\lws,\lwt):=\sii$ defined as above.
\end{definition}
We have the following equations among bi-indices, which will be used later.
\begin{lemma}[{\cite[Lemma 2.6]{IQZ}}]\label{lem:biindsum}
Let $\lws,\lwt$ be bigraded arcs in $\DMS$ with an intersection $p\in\isurf$. If $p\notin\Tri$, we have
\begin{equation}\label{eq:hkkg}
	\dind_p(\lws,\lwt)+\dind_p(\lwt,\lws)=1.
\end{equation}
If $p\in\Tri$, we have
\begin{equation}\label{eq:hkkg1}
	\dind_p(\lws,\lwt)+\dind_p(\lwt,\lws)=\XX.
\end{equation}
\end{lemma}
\begin{lemma}[{\cite[Lemma 2.7]{IQZ}}]\label{lem:bi3ind}
Let $\lws,\lwt,\lwa$ be bigraded arcs in $\DMS$. If they are in the case in the left one of Figure~\ref{fig:3int1}, we have
\begin{equation}\label{eq:3int}
\dind_p(\lws,\lwa)=\dind_{p}(\lws,\lwg)+\dind_{p}(\lwg,\lwa).
\end{equation}
If they are in the case in the left one of Figure~\ref{fig:iid}, we have
\begin{align}\label{eq:ksg}
\begin{split}
\dind_p(\lws,\lwt)&=\dind_{\alpha(0)}(\lws,\lwa)-\dind_{\alpha(1)}(\lwt,\lwa)\\
		&=\dind_{\alpha(1)}(\lwa,\lwt)-\dind_{\alpha(0)}(\lwa,\lws).
\end{split}
\end{align}
\begin{figure}[h]\centering
	\begin{tikzpicture}[scale=.6]
	\foreach \j in {0,1,2}{
		\draw[red,thick] (180+120*\j:3)to(120*\j:3);}
	\draw[red,](120:3)to(-60:3)node[right]{$\lwt$} (0,0)node{$\bullet$} (3,0)node[right]{$\lwa$}(-120:3)node[left]{$\lws$};
	\draw[blue!50,thick,-<-=.7,>=stealth]	(0:1.2)to[bend left=60](-120:1.2);
	\draw[blue!50,thick,-<-=.7,>=stealth]	(0:.9)to[bend left=15](-60:.9);
	\draw[blue!50,thick,-<-=.7,>=stealth]	(-60:.7)to[bend left=15](-120:.7);
	\draw[red] (0,0)node[above]{$p$};
\begin{scope}[shift={(8,0)}]
\foreach \j in {0,1,2}{\draw[red, thick] (240+60*\j:0)to(240+60*\j:3);}
\draw[white] (0,0)node{$\bullet$};
\draw[red,](300:3)node[right]{$\lwt$} (0,0)\ww (3,0)node[right]{$\lwa$}(-120:3)node[left]{$\lws$};
\draw[blue!50,thick,-<-=.7,>=stealth]	(0:1.1)arc(0:-120:1.1)(-120:1.1);
\draw[blue!50,thick,-<-=.7,>=stealth](0:.85)arc(0:-60:.85)(-60:.85);
\draw[blue!50,thick,-<-=.7,>=stealth](-60:.7)arc(-60:-120:.7)(-120:.7);
\draw[red] (0,0)node[above]{$\d$} (290:1.1)node[below]{$\lwth$};
\end{scope}
	\end{tikzpicture}
	\caption{Bigraded arcs intersect at the same point (or decoration) in anti-clockwise}\label{fig:3int1}
\end{figure}
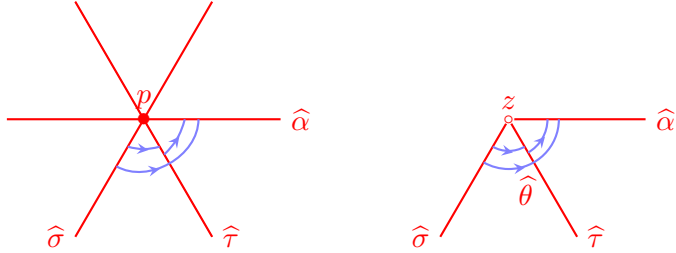
\end{lemma}
\begin{lemma}\label{lem:3intindec}
Let $\lws,\lwt,\lwa$ be bigraded arcs on $\DMS$ which share the same decoration $z$ and sitting in anti-clockwise order in the right one of Figure~\ref{fig:3int1}. We have
\begin{equation}\label{eq:3intindec}
\dind_{\d}(\lws,\lwa)=\dind_{\d}(\lws,\lwt)+\dind_{\d}(\lwt,\lwa).
\end{equation}
\end{lemma}
\begin{proof}
Fix a small circle $a\subset\surf\setminus\Tri$ around $z$. Let $\lwth:[0,1]\to a$ be an embedded bigraded arc winding anti-clockwisely at $z$ such that the underlying arc $\theta$ intersect $\us, \ut$ and $\ua$ at $\theta(0), \theta(\frac{1}{2})$ and $\theta(1)$ respectively (see the right one of Figure \ref{fig:3int1}). The arc $\lwth$ is unique up to a change of parametrization. By \Cref{lem:bi3ind}, we have
\begin{equation*}
\begin{split}
\dind_z(\lws,\lwa)&=\dind_{\theta(0)}(\lws,\lwth)-\dind_{\theta(1)}(\lwa,\lwth)\\&=[\dind_{\theta(0)}(\lws,\lwth)-\dind_{\theta(\frac{1}{2})}(\lwt,\lwth)]+[\dind_{\theta(\frac{1}{2})}(\lwt,\lwth)-\dind_{\theta(1)}(\lwa,\lwth)]\\&=\dind_{z}(\lws,\lwt)+\dind_{z}(\lwt,\lwa).
\end{split}
\end{equation*}
\end{proof}

\begin{definition}\label{def:smoothout}
For $\lwt,\lwe\in\wXCA(\surfo)$ satisfying that $\Int_{\surfo}(\ut,\ue)=\frac{1}{2}$ and $\sind_z(\lwt,\lwe)=a$,
let $z\in\Tri$ be their common endpoint.
Denote by $\lwte$ to be the bigraded arc in $\wXCA(\surfo)$ whose underlying arc is obtained by the smoothing out $\lwt\cup\lwe[(a-1)\XX]$ at $z$
and whose grading inherits from $\lwt$. That is, we have $\dind_z(\lwt,\lwte)=0$, cf. Figure \ref{fig:indsum}.
\end{definition}

\begin{proposition}\label{prop:indsum}
For $\lwt, \lwe$ and $\lwte$ in $\wXCA(\surfo)$ as in \Cref{def:smoothout}, we have
\[\dind_{\de}(\lwt,\lwte)+\dind_{\dy}(\lwte,\lwe)+\dind_{\ds}(\lwe,\lwt)=1.\]
\end{proposition}
\begin{proof}
\begin{figure}[h]\centering
\begin{tikzpicture}[scale=.8]
\draw[red,very thick,>=stealth](0,0)to[bend right=13](150:3);
\draw[red,very thick,>=stealth](0,0)to[bend left=13](30:3);
\draw[red,very thick,>=stealth](30:3)to[bend right=13](150:3);
\draw[white](30:3)\nn(150:3)\nn (0,0)\nn;
\draw[red](30:3)\ww(150:3)\ww (0,0)\ww;
\draw[blue!50,-<-=0.5,>=stealth,thick](140:.5)to[bend left=30](40:.5);
\draw[blue!50,-<-=0.7,>=stealth,thick]($(150:3)+(10:.7)$)to[bend left=30]($(150:3)+(-20:.7)$);
\draw[blue!50,-<-=0.7,>=stealth,thick]($(30:3)+(200:.7)$)to[bend left=30]($(30:3)+(170:.7)$);
\draw[red](150:3)node[left]{$\de$} (30:3)node[right]{$\dy$} (0,0)node[below]{$\ds$} (0,2.3)node{$\lwte$} (150:1)node[left]{$\lwt$} (30:1)node[right]{$\lwe$};
\end{tikzpicture}
\caption{The sum of bi-indices of intersections between bigraded arcs via smoothing out} \label{fig:indsum}
\end{figure}
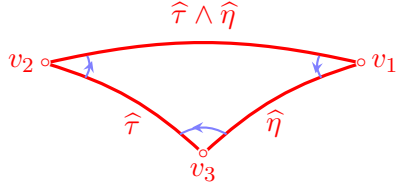
We calculate the two gradings separately.
For the first grading, it is a line field which is determined by $\ho{1}(\PTSt,\ZZ)$.
We can identify all the projectization of the tangent space of any point in the contractible triangle formed by the three arcs
(except the decorations) simultaneously.
Hence, the sum of the first grading is 1 (rotating anti-clockwisely).
For the second grading, we use the log surface in the sense of \cite[\S 2.4]{IQZ}.
By \Cref{def:smoothout}, the segments of $\lwte$ and $\lwe[(a-1)\XX]$ near $\dy$ are in the same sheet of $\log(\surfo)$
and the anti-clockwise angle does not cross the cut (cf. \cite[Figure 5]{IQZ}).
Thus we have
\begin{equation*}
\begin{split}
\ind^{\XX}_{\de}(\lwt,\lwte)+\ind^{\XX}_{\dy}(\lwte,\lwe)+\ind^{\XX}_{\ds}(\lwe,\lwt)&=0+(a-1)\XX+(1-\ind^{\XX}_{\ds}(\lwt,\lwe))\\
    &=(a-1)+(1-a)=0
\end{split}
\end{equation*}
as required, where $\ind_{?}^{\XX}$ denote the second grading.
\end{proof}
\begin{notations}
For $(\ii,\sii)\in\ZZ^2$, we denote $\cap^{\ii+\sii\XX}(\lws,\lwt)$ the set of intersections between $\lws$ and $\lwt$ with bi-index $\ii+\sii\XX$. We will use the notations \begin{gather*}
\Int^{\ii+\sii\XX}_{\Pi}(\lws,\lwt)\colon=\left|\cap^{\ii+\sii\XX}(\lws,\lwt)\cap \Pi\right|,\\
\Int^{\ii+\sii\XX}_{\surfo}(\lws,\lwt)\colon=
\frac{1}{2}\cdot\Int^{\ii+\sii\XX}_{\Tri}(\lws,\lwt)+
\Int^{\ii+\sii\XX}_{\surfoi}(\lws,\lwt),
\end{gather*}
for the bi-index $(\ii+\sii\XX)$ intersection numbers at any proper subset $\Pi$ of $\surfo$ and at all of $\isurf$ respectively.
The \emph{total intersection}
$$\Int_{?}(\lws,\lwt)=\sum_{\ii,\sii\in\ZZ}\Int^{\ii+\sii\XX}_?(\lws,\lwt)$$
is the sum over all bi-indices, where $?=\Tri$ or $\surfoi$.
\end{notations}

\begin{definition}[{\cite{KS,IQZ}}] 
Let $\qz$ and $\qx$ be two formal parameters. The \emph{$\ZZ^2$-graded \txq-intersection} of $\lws,\lwt\in\wXCA(\surfo)$ is defined to be
	\begin{equation}\label{eq:q-int}
	\qqInt(\lws,\lwt)=
	\displaystyle\sum_{\ii,\sii\in\mathbb{Z}}
	\qz^{\ii}\qx^{\sii}\cdot\Int_{\Tri}^{\ii+\sii\XX} (\lws,\lwt)
	+(1+\qz^{-1}\qx)
	\displaystyle\sum_{\ii,\sii\in\mathbb{Z}}
	\qz^{\ii}\qx^{\sii}\cdot\Int_{\surfoi}^{\ii+\sii\XX} (\lws,\lwt).
	\end{equation}
	Note that we have $\qqInt(-,-)\mid_{\qz=\qx=1}=2\Int_{\surfo}(-,-)=\Int_{\bc}(-,-)$, where $\bc$ is the branched double cover of $\surfo$.
\end{definition}
\subsection{Left \txq-deformations as \txq-intersections}\label{sec:left}
Recall that $\disk$ is a disk and $\Tri=\{\deci,\decs,\decz\}$. By Lemma \ref{lem:CA}, we can label simple closed arcs by $\uQ$ as follows. We fix two initial bigraded simple closed arcs and denote them $\lwez$ and $\lwei$ such that $\ind_{\decs}(\lwei,\lwez)=1$ (see Figure \ref{fig:closedarc}).  
\begin{construction}
We define a map
\begin{equation}
\begin{split}
\lwe: \uQ&\to\wXCA(\disk)\\
\frac{r}{s}&\mapsto\lwq,
\end{split}
\end{equation}
as follows. For any rational $\rat\in\ugQ$, by \Cref{rmk:deco} we know that it can be uniquely written as
\[\rat=\frac{p}{q}\oplus\frac{u}{v}.\] We iteratively define that
\begin{equation}
\lwq:=B_{\ue_{\frac{u}{v}}}(\lwe_{\frac{p}{q}})=\widehat{\ue}_{\frac{p}{q}}\wedge\widehat{\ue}_{\frac{u}{v}},
\end{equation}
noticing that $\Int_{\disk}(\ue_{\frac{p}{q}},\ue_{\frac{u}{v}})=\frac{1}{2}$. For negative case, we set $\lwe_{-\infty}:=\lwei$ and define
\begin{equation}
\lwe_{-\frac{r}{s}}:=B_{\ue_{-\frac{p}{q}}}(\lwe_{-\frac{u}{v}}).
\end{equation}
Thus, we get some bigraded closed arcs in $\wXCA(\disk)$ which are $\lwq[\ii+\sii\XX]$, where $\ii,\sii\in\ZZ$ and $\frac{r}{s}\in\uQ$.
\end{construction}
Let $\frac{r}{s}\in\QQ^+$ with $\frac{r}{s}=\frac{p}{q}\oplus\frac{u}{v}$. By \Cref{def:smoothout}, the grading of the new arc $\lwqs$ inherits the grading of $\lwp$. That is, for any bigraded arc $\lws$ intersecting $\lwq$ and $\lwqs$ at $p_1, p_2\in\isurf$ respectively (see Figure \ref{fig:inherit}), we have
\begin{equation}\label{inherit}
\dind_{z}(\lwp,\lwqs)=0,\quad \dind_{p_1}(\lws,\lwp)=\dind_{p_2}(\lws,\lwqs),
\end{equation}
where $z\in\Tri$ is the common endpoint of $\lwp$ and $\lwqs$. The second equation in \eqref{inherit} follows from the fact that $\lwp,\lwqs$ and $\lwqq$ form a contractible triangle.
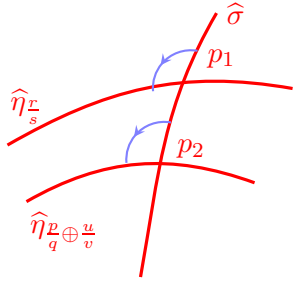
\begin{figure}[ht]\centering
\begin{tikzpicture}[scale=.5]
\draw[red, very thick](-4,-3)to[out=30,in=160](2,-2.5);
\draw[red, very thick](-4.5,-1.5)to[out=30,in=170](3,0);
\draw[red, very thick](-1,-5)to[out=80,in=240](1,2);
\draw(-3,-3)node[below,red]{$\lwqs$} (-4,-1.2)node[above,red]{$\lwq$} (1,2)node[right,red]{$\lws$} (.5,.8)node[right,red]{$p_1$} (-.3,-1.6)node[right,red]{$p_2$};
\draw[blue!50,thick,->-=.6,>=stealth,shift={(-.7,-1.8)}]	(.5,.9) arc (80:185:1);
\draw[blue!50,thick,->-=.6,>=stealth]	(.5,1) arc (80:185:1);
\end{tikzpicture}
\caption{The bigrading inherited from braid twist}\label{fig:inherit}
\end{figure}

From \Cref{prop:indsum}, we know that these three simple closed arcs $\lwp,\lwqs$ and $\lwqq$ (take $\lwt=\lwp, \lwe=\lwqq$ in Figure \ref{fig:indsum}) satisfy
\begin{equation}\label{eq:indsum}
\dind_{\de}(\lwp,\lwqs)+\dind_{\dy}(\lwqs,\lwqq)+\dind_{\ds}(\lwqq,\lwp)=1,
\end{equation}
where $\dy,\de,\ds$ are the corresponding intersecting decorations.

We fix the following setting.
\begin{setting}\label{set}
Recall that in \Cref{rmk:deco}, any fraction $\frac{r}{s}\in\QQ^+$ can be uniquely written as
\[\rat=\frac{p}{q}\oplus\frac{u}{v},\]
with $\frac{p}{q},\frac{u}{v}\in\upQ,uq-pv=1$ and an associated integer $l(\frac{p}{q},\frac{u}{v})$. The corresponding arcs in $\wXCA(\disk)$ of these fractions are $\lwp, \lwqs, \lwqq$, where $\lwp$ and $\lwqq$ intersect at only one decoration $z$ in $\Tri$.
We do not distinguish $\frac{r}{s}$ and $\frac{p}{q}\oplus\frac{u}{v}$ in the following.
\end{setting}
\begin{lemma}\label{lem:ind}
For any two arcs $\lwp, \lwqq$ in $\wXCA(\disk)$ in \Cref{set}, the bi-index is of the form $l(1-\XX)$ where $l\in\NN$, i.e.
\begin{equation}\label{eq:ind}
\dind_z(\lwp, \lwqq)=l(1-\XX),
\end{equation}
except the special case
\begin{equation}\label{eq:indintitial}
\dind_{\decs}(\lwez, \lwei)=\XX-1.
\end{equation}
Here $l=l(\frac{p}{q},\frac{u}{v})$ in \Cref{rmk:deco}.
\end{lemma}
\begin{proof}
For the special case, we know that
\begin{equation*}
\dind_{\decs}(\lwez, \lwei)=\XX-\dind_{\decs}(\lwei,\lwez)=\XX-1,
\end{equation*}
from \eqref{eq:hkkg1}. In general, we prove the lemma by induction on $l$. For initial bigraded simple closed arcs $\lwez, \lwe_{\frac{1}{1}}$ and $\lwei$ in $\wXCA(\disk)$, we have $\ind_{\decz}(\lwez, \lwe_{\frac{1}{1}})=\ind_{\deci}(\lwe_{\frac{1}{1}},\lwei)=0$ and thus the lemma holds obviously. We assume that \eqref{eq:ind} holds for $l$ and consider the $l+1$ case. For any $\lwp, \lwqq$ in $\wXCA(\disk)$ in \Cref{set}, we assume that they intersect at only one decoration $\ds$ with
\begin{equation}
\dind_{\ds}(\lwp, \lwqq)=l(1-\XX),
\end{equation}
where $l\in\NN$. By \eqref{eq:hkkg1}, we have
\[\dind_{\ds}(\lwqq,\lwp)=(l+1)\XX-l.\]
Since the grading of $\lwqs$ inherits the grading of $\lwp$, we have
\begin{equation}\label{eq:indqs}
\dind_{\dy}(\lwp,\lwqs)=0.
\end{equation}
By \eqref{eq:indsum}, we deduce that
\begin{equation}\label{eq:indsqq}
\begin{split}
\dind_{\de}(\lwqs,\lwqq)&=1-0-[(l+1)\XX-l]\\&=(l+1)-(l+1)\XX.
\end{split}
\end{equation}
Finally, combining \eqref{eq:indqs} and \eqref{eq:indsqq}, the lemma is true.
\end{proof}
\begin{theorem}\label{thm:leftint}
For any rational number $\frac{r}{s}\in\uQ$, we have
\begin{equation}\label{eq:leftint}
\Big[\rat\Big]^{\flat}_{\q}=\frac{\varepsilon\Int^{\q}(\lwq,\lwez)}{\Int^{\q}(\lwq,\lwei)}\bigg|_{\q=\qz^{-1}\qx},
\end{equation}
corresponding to the left \txq-deformation of $\rs$, where
\begin{equation*}
\varepsilon=
\begin{cases}
\qz^{-1},&\frac{r}{s}\geq 0,\\
-1,&\frac{r}{s}<0.
\end{cases}
\end{equation*}
In particular, for $\frac{r}{s}\in\upQ$, we have
\emph{
\begin{align}
\begin{split}
\left \{
\begin{array}{ll}
\qrb{\rs}&=\qz^{-1}\Int^{\q}(\lwq,\lwez)\big|_{\q=\qz^{-1}\qx},\\
\qsb{\rs}&=\Int^{\q}(\lwq,\lwei)\big|_{\q=\qz^{-1}\qx}.
\end{array}
\right.
\end{split}
\end{align}}

\begin{figure}[ht]\centering
\makebox[\textwidth][c]{
	\begin{tikzpicture}[scale=.47,yscale=.8,xscale=-1,font=\tiny]
	\draw[very thick](0,0)circle(6);
\draw(-3,0)edge[dashed,red, very thick](0,0);
\draw(0,0)edge[dashed,red, very thick](3,0);
\draw[red, very thick](-3,0) to[out=40,in=180](-.2,1.2)to[out=0,in=180](3,-1.2)to[out=0,in=260](4,0)to[out=80,in=0](0,4) to[out=180,in=90](-4.5,0) to[out=270,in=180](-2,-2)to[out=0,in=280](0,0);
\draw(-3,0)node[below,black]{$\deci$} (-.5,0)node[below,black]{$\decs$} (3,0)node[below,black]{$\decz$};
\draw (2,3)node[below,red]{$\lwe_{\frac{3}{2}}$};
	\draw(-3,0)node[white] {$\bullet$} node[red] {$\circ$};
	\draw(0,0)node[white] {$\bullet$} node[red] {$\circ$};
	\draw(3,0)node[white] {$\bullet$} node[red] {$\circ$};
 \draw(0,6)node[cyan] {$\bullet$} (250:6)node[cyan] {$\bullet$} (290:6)node[cyan] {$\bullet$};
\draw(0,6)edge[cyan, very thick](250:6) (0,6)edge[cyan, very thick](290:6);
\draw(-2.4,-3.6)node[cyan]{$\wgi$} (2.4,-3.6)node[cyan]{$\wgz$};
 \begin{scope}[shift={(-13,0)}]
\draw[very thick](0,0)circle(6);
\draw(-3,0)edge[dashed, red, very thick](0,0);
\draw(0,0)edge[dashed,red, very thick](3,0);
\draw[red, very thick,yscale=-1](3,0) to[out=90,in=0](0,3)to[out=180,in=90](-4,0) to[out=270,in=180](-2.5,-2)to[out=0,in=270](0,0);
\draw(-3,0)node[above,black]{$\deci$} (0,0)node[below,black]{$\decs$} (3,0)node[above,black]{$\decz$};
\draw (-1.8,2.7)node[red]{$\lwe_{-\frac{2}{1}}$};
	\draw(-3,0)node[white] {$\bullet$} node[red] {$\circ$};
	\draw(0,0)node[white] {$\bullet$} node[red] {$\circ$};
	\draw(3,0)node[white] {$\bullet$} node[red] {$\circ$};
 \draw(0,6)node[cyan] {$\bullet$} (250:6)node[cyan] {$\bullet$} (290:6)node[cyan] {$\bullet$};
\draw(0,6)edge[cyan, very thick](250:6) (0,6)edge[cyan, very thick](290:6);
\draw(-2.4,-3.6)node[cyan]{$\wgi$} (2.4,-3.6)node[cyan]{$\wgz$};
 \end{scope}
 \begin{scope}[shift={(13,0)}]
	\draw[very thick](0,0)circle(6);
\draw(-3,0)edge[dashed, red, very thick](0,0);
\draw(0,0)edge[dashed,red, very thick](3,0);
\draw[red, very thick](-3,0)to[out=80,in=180](0,3)to[out=0,in=100](3,0);
\draw(-3,0)node[below,black]{$\deci$} (0,0)node[below,black]{$\decs$} (3,0)node[below,black]{$\decz$};
\draw(-1.5,0)node[above,red]{$\lwei$} (1.5,0)node[above,red]{$\lwez$} (0,1.5)node[above,red]{$\lwe_1$};
	\draw(-3,0)node[white] {$\bullet$} node[red] {$\circ$};
	\draw(0,0)node[white] {$\bullet$} node[red] {$\circ$};
	\draw(3,0)node[white] {$\bullet$} node[red] {$\circ$};
 \draw(0,6)node[cyan] {$\bullet$} (250:6)node[cyan] {$\bullet$} (290:6)node[cyan] {$\bullet$};
\draw(0,6)edge[cyan, very thick](250:6) (0,6)edge[cyan, very thick](290:6);
\draw(-2.4,-3.6)node[cyan]{$\wgi$} (2.4,-3.6)node[cyan]{$\wgz$};
 \end{scope}
	\end{tikzpicture}
}
\caption{Red arcs: examples of bigraded closed arcs in $\wXCA(\disk)$}\label{fig:closedarc}
\end{figure}
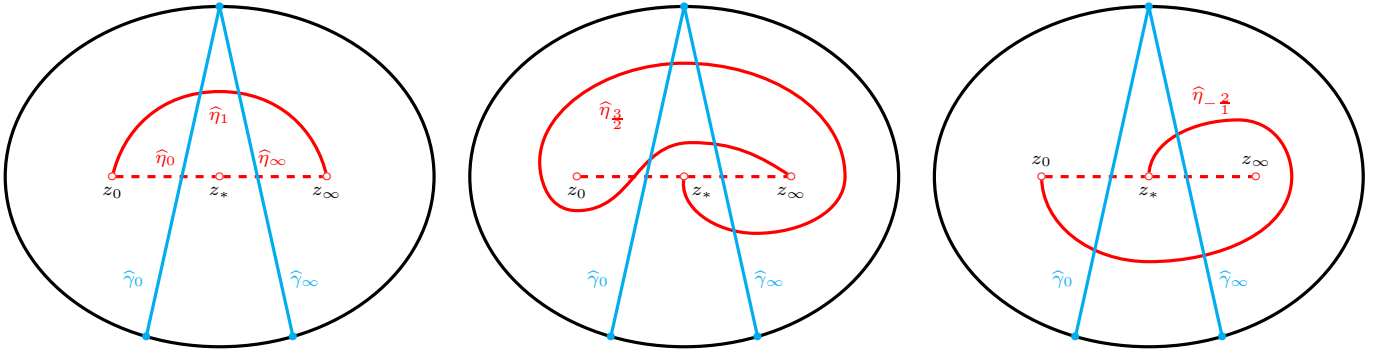
\end{theorem}
\begin{proof}
For $0$ and $\infty$, we have
\begin{equation*}
\Int^{\q}_{\decs}(\lwez,\lwei)=\qz^{-1}\qx~{\rm and}~\Int^{\q}_{\decs}(\lwei,\lwez)=\qz
\end{equation*}
by definition. We set
\begin{equation*}
\qrb{\zero}=\qx-\qz~{\rm and}~\qsb{\infi}=1-\qz^{-1}\qx.
\end{equation*}
We prove the non-negative case by induction using \Cref{set}. At starting step, the fractions in the theorem are $\frac{\qz^{-1}\qx-1}{\qz^{-1}\qx}$ and $\frac{1}{1-\qz^{-1}\qx}$, which coincide the left $\q$-deformation of $0$ and $\infi$ respectively when $\q=\qz^{-1}\qx$ We assume that the formula \eqref{eq:leftint} holds for $\frac{p}{q},\frac{u}{v}\in\upQ$ in \Cref{set} which satisfies
\begin{equation*}
\dind_z(\lwp, \lwqq)=l(1-\XX),
\end{equation*} where $l=l(\frac{p}{q}\oplus\frac{u}{v})\in\NN$ and $z$ is the common endpoint of $\lwp$ and $\lwqq$. We use the $\ZZ^2$-covering $\ubc$ to compute the \txq-intersection between $\lwqs$ and $\lwez$~(and $\lwei$ is similar). We lift the arcs in $\wXCA(\disk)$ to lines in $\ubc$. Then $\lwez$ becomes a series of horizontal lines which pass through $\decs$ and $\decz$. The topological triangle $T$ in $\disk$ bounded by $\lwp,\lwqq$ and $\lwqs$ becomes a triangle $\widetilde{T}$ in $\RR^2$ up to translation (or reflection). We may assume that the vertices of $\widetilde{T}$ are $\widetilde{z},\widetilde{z}+(p,q), \widetilde{z}+(r,s)$, which says that $z', z, z''$ in $T$. The area of $\widetilde{T}$ equals to $\frac{1}{2}$ by the relation $uq-pv=1$. By Pick's theorem, there are no decorations or punctures in $\widetilde{T}$. The intersections between $\lwqs$ and $\lwez$ consist of two parts (see Figure \ref{fig:indinh}):
\begin{itemize}
\item intersections below and on $a$, which inherit from $\lwp$, and
\item intersections above $a$, which is induced from $\lwqq$.
\end{itemize}
\begin{figure}[ht]\centering
\begin{tikzpicture}[scale=.6]
\draw[dashed, very thick,>=stealth] (-2,0)--(2,0) (0,2)--(4,2) (1,3)--(5,3) (2,6)--(6,6) (4,9)--(8,9);
\draw[red,very thick,>=stealth] (0,0)--(4,3)--(6,9) (0,0)--(6,9);
\draw[white](0,0)\nn(4,3)\nn (6,9)\nn;
\draw[red](0,0)\ww(4,3)\ww (6,9)\ww;
\draw (2,1)node[right]{$\lwp$} (5.5,7)node[right]{$\lwqq$} (3,5)node[left]{$\lwqs$} (2,0)node[right]{$\lwez$} (0,0)node[below]{$z'$} (4,3)node[below]{$z$} (6,9)node[above]{$z''$} (5,3)node[right]{$a$} (2,3)node[below]{$p$};
\end{tikzpicture}
\caption{Intersections whose bi-indices inherit (resp. are induced from) the one between $\lwp$ (resp. $\lwqq$) and $\lwez$}\label{fig:indinh}
\end{figure}
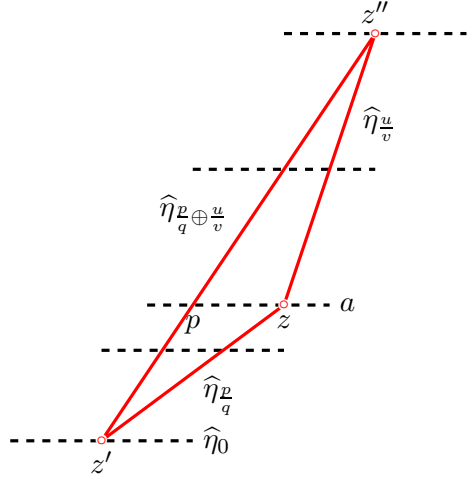
For the first case, if the two intersections are both either in the interior or at decorations, the bi-indexes are same. For the second case, the two bi-indexes differ from $\dind_{z''}(\lwqs,\lwqq)=(l+1)\cdot(1-\XX)$. Thus we have
\begin{equation}
\Int^{\q}_{\disk\backslash\{p\}}(\lwqs,\lwez)=\Int^{\q}_{\disk\backslash\{z\}}(\lwp,\lwez)+\qz^{l+1}\qx^{-l-1}\cdot\Int^{\q}_{\disk\backslash\{z\}}(\lwqq,\lwez).
\end{equation}
For the intersection on $a$, by inheritance, we have bi-index given directly by
\begin{equation}
\dind_{\pn}(\lwez,\lwqs)=\dind_{\pt}(\lwez,\lwp).
\end{equation}
Hence we have
\begin{equation*}
\begin{split}
\dind_{\pn}(\lwqs,\lwez)&=1-\dind_{\pn}(\lwez,\lwqs)=1-\dind_{\pt}(\lwez,\lwp)\\&=1-\XX+\dind_{\pt}(\lwp,\lwez).
\end{split}
\end{equation*}
Thus, the bi-index in $\diski$ contributes $(1-\XX)+\dind_{\pt}(\lwp, \lwez)$ and $\dind_{\pt}(\lwp, \lwez)$ in \txq-intersection from \eqref{eq:q-int}.

On the other hand, we have
\begin{equation*}
\begin{split}
\dind_{\pt}(\lwqq,\lwez)&=\dind_{\pt}(\lwp,\lwez)-\dind_{\pt}(\lwp,\lwqq)\\&=[\dind_{\pt}(\lwp,\lwez)+(1-\XX)]-[\dind_{\pt}(\lwp,\lwqq)+(1-\XX)]\\&=[\dind_{\pt}(\lwp,\lwez)+(1-\XX)]-(l+1)\cdot(1-\XX).
\end{split}
\end{equation*}
Thus, we deduce that
\begin{equation}
\Int^{\q}_z(\lwqs,\lwez)=\Int^{\q}_z(\lwp,\lwez)+\qz^{l+1}\qx^{-l-1}\cdot\Int^{\q}_z(\lwqq,\lwez).
\end{equation}
Therefore, we have
\begin{equation}
\Int^{\q}(\lwqs,\lwez)=\Int^{\q}(\lwp,\lwez)+\qz^{l+1}\qx^{-l-1}\cdot\Int^{\q}(\lwqq,\lwez),
\end{equation}
which coincides with $\qrb{\rs}$ with a multiplication of $\qz^{-1}$ if we take $\q=\qz^{-1}\qx$.
Similarly, we deduce that
\begin{equation}
\Int^{\q}(\lwqs,\lwei)=\Int^{\q}(\lwp,\lwei)+\qz^{l+1}\qx^{-l-1}\cdot\Int^{\q}(\lwqq,\lwei).
\end{equation}
If we take $\qz^{-1}\qx=\q$, the fraction we get in the theorem coincides the left \txq-deformation of $\frac{r}{s}=\frac{p}{q}\oplus\frac{u}{v}$, which completes the proof.
\end{proof}
\begin{example}
We give examples of $\frac{3}{2}$ and $-2$ (see Figure \ref{fig:closedarc}) for the left $\q$-deformations. We compute that
\begin{align*}
\begin{split}
\left \{
\begin{array}{ll}
\Int^{\q}(\lwe_{\frac{3}{2}},\lwez)&=\qz+\qx+\qz^3\qx^{-2},\\
\Int^{\q}(\lwe_{\frac{3}{2}},\lwei)&=1+\qz^2\qx^{-2};
\end{array}
\right.
\end{split}
\end{align*}
and
\begin{align*}
\begin{split}
\left \{
\begin{array}{ll}
\Int^{\q}(\lwe_{-\frac{2}{1}},\lwez)&=\qz^3\qx^{-2}+\qz,\\
\Int^{\q}(\lwe_{-\frac{2}{1}},\lwei)&=\qx.
\end{array}
\right.
\end{split}
\end{align*}
By \Cref{thm:leftint}, we have
\begin{equation*}
\Big[\frac{3}{2}\Big]^{\flat}_{\q}=\frac{\qz^{-1}\Int^{\q}(\lwe_{\frac{3}{2}},\lwez)}{\Int^{\q}(\lwe_{\frac{3}{2}},\lwei)}\bigg|_{\q=\qz^{-1}\qx}=\frac{\q^3+\q^2+1}{\q^2+1},
\end{equation*}
and
\begin{equation*}
\Big[-\frac{2}{1}\Big]^{\flat}_{\q}=-\frac{\Int^{\q}(\lwe_{-\frac{2}{1}},\lwez)}{\Int^{\q}(\lwe_{-\frac{2}{1}},\lwei)}\bigg|_{\q=\qz^{-1}\qx}=-\frac{\q^2+1}{\q^3}.
\end{equation*}
\end{example}
\subsection{Right \txq-deformations as \txq-intersections}\label{sec:right}
In this subsection, we add a finite set $\M$ of \emph{(open) marked points} to $\bdy$ satisfying $|\M|=|\Tri|$ and get a decorated marked surface (or \emph{DMS} for short). We still denote the DMS by $\surfo$.
An arc $c$ is called \emph{open} if $c(0)$ and $c(1)$ are in $\M$, without self-intersections in $\surfoi$. We call two open arcs do not cross each other if they do not have intersections in $\surfoi$.

We also have bigraded open arcs as before. For an open arc $\gamma$, we define the \emph{$\ZZ^2$-graded \txq-intersection} between a lift $\lwg$ of $\gamma$ and $\lwt\in\wXCA(\surfo)$ to be
	\begin{equation}\label{eq:q-int2}
	\qqInt(\lwg,\lwt)=
	\displaystyle\sum_{\ii,\sii\in\mathbb{Z}}
	\qz^{\ii}\qx^{\sii}\cdot\Int_{\surfoi}^{\ii+\sii\XX} (\lwg,\lwt).
	\end{equation}
 Note that we have $\qqInt(\lwg,\lwt)\mid_{\qz=\qx=1}=\Int_{\surfo}(\ug,\ut)$. We define a special class of open (bigraded) arcs.
\begin{definition}\label{def:f.f.a.s.+dual.a.s.}
An \emph{open full formal arc system} $\A=\{\ug_1,\cdots,\ug_n \}$ of a DMS $\surfo$ is a collection of pairwise non-crossing open arcs, such that it divides the surface $\surfo$ into polygons, called \emph{$\A$-polygons}, satisfying that each $\A$-polygon contains exactly one decoration. We call $\us\in\CA(\surfo)$ the \emph{dual} to $\ug_i$, if $\ug_i$ intersects it once and $\ug_j$ does not intersect it for any $j\neq i$. Denote $s_i$ the dual to $\ug_i$ and $\A^*=\{s_1,\cdots,s_n\}$.
\end{definition}
Let $\wg_1,\cdots,\wg_n, \widehat{s}_1,\cdots,\widehat{s}_n$ be their bigraded lifts with
\begin{equation}\label{eq:dual}
\dind(\wg_i,\widehat{s}_j)=\delta_{ij}.
\end{equation}

We add three open marked points to the boundary of $\disk$ in Section \ref{sec:left}. Let $\ug_{\zero}$, $\ug_{\infi}$ be two open arcs which form an open arc system in $\disk$ and intersect with $\ue_{\zero}$ and $\ue_{\infi}$ transitively only once respectively. Let $\wgz, \wgi$ be their bigraded lifts respectively which satisfy
\begin{equation}
\dind(\wgi,\lwei)=\dind(\wgz,\lwez)=0.
\end{equation}
We draw them as blue arcs in Figure \ref{fig:closedarc}.
\begin{theorem}\label{thm:rightint}
For any rational number $\frac{r}{s}\in\uQ$, we have
\begin{equation}\label{eq:rightint}
\Big[\rat\Big]^{\sharp}_{\q}=\frac{\varepsilon\Int^{\q}(\wgi, \lwq)}{\Int^{\q}(\wgz, \lwq)}\bigg|_{\q=\qz^{-1}\qx},
\end{equation}
corresponding to the right \txq-deformation of $\rs$, where
\begin{equation*}
\varepsilon=
\begin{cases}
1,&\frac{r}{s}\geq0,\\
-\qz^{-1},&\frac{r}{s}<0,
\end{cases}
\end{equation*}
and the polynomials in the numerator and denominator are polynomials in $\ZZ[\qz^{-1}\qx]$. In particular, for $\frac{r}{s}\in\upQ$, we have
\emph{
\begin{align}
\begin{split}
\left \{
\begin{array}{ll}
\qr^{\sharp}(\rs)&=\Int^{\q}(\wgi, \lwq)\big|_{\q=\qz^{-1}\qx},\\
\qs^{\sharp}(\rs)&=\Int^{\q}(\wgz, \lwq)\big|_{\q=\qz^{-1}\qx}.
\end{array}
\right.
\end{split}
\end{align}}
\end{theorem}
\begin{proof}
The theorem follows the same way of the left version and we prove the non-negative case by induction using \Cref{set}. For the starting case, we have
\begin{equation*}
\begin{split}
\Int(\wgi, \lwez)&=\Int(\wgz, \lwei)=0,\\ \Int(\wgi, \lwei)&=\Int(\wgz, \lwez)=\qz^0\qx^0=1.
\end{split}
\end{equation*}
Thus, they coincide with the right \txq-deformation of 0 and $\infty$ respectively. We assume that the formula \eqref{eq:rightint} holds for $\frac{p}{q},\frac{u}{v}\in\upQ$ in \Cref{set} which satisfy
\begin{equation*}
\dind_z(\lwp, \lwqq)=l(1-\XX),
\end{equation*} where $l=l(\frac{p}{q}\oplus\frac{u}{v})\in\NN$. Similar as the left version, the intersections between $\wgz$ and $\lwqs$ consist of two parts. One inherits from $\lwp$, whose bi-indices are same; and the other one is induced from $\lwqq$, whose bi-indices differ from $-\dind_{z''}(\lwqs,\lwqq)=(l+1)\cdot(\XX-1)$. Notice that the intersections are all in the interior, which simplifies things a lot. Thus, we have
\begin{equation}\label{eq:rightformula}
\Int^{\q}(\wgz,\lwqs)=\Int^{\q}(\wgz,\lwp)+\qz^{-l-1}\qx^{l+1}\cdot\Int^{\q}(\wgz,\lwqq),
\end{equation}
which coincides denominator of the right \txq-deformation of $\frac{p}{q}\oplus\frac{u}{v}$ if we take $\qz^{-1}\qx=\q$. Similarly, we deduce that
\begin{equation}
\Int^{\q}(\wgi,\lwqs)=\Int^{\q}(\wgi,\lwp)+\qz^{-l-1}\qx^{l+1}\cdot\Int^{\q}(\wgi,\lwqq).
\end{equation}
Thus, we finish the proof.
\end{proof}
\begin{example}
We continue the examples of $\frac{3}{2}$ and $-2$ (see Figure \ref{fig:closedarc}) for the right $\q$-deformations. We compute that
\begin{align*}
\begin{split}
\left \{
\begin{array}{ll}
\Int^{\q}(\wgi, \lwe_{\frac{3}{2}})&=1+\qz^{-1}\qx+\qz^{-2}\qx^2,\\
\Int^{\q}(\wgz, \lwe_{\frac{3}{2}})&=1+\qz^{-1}\qx;
\end{array}
\right.
\end{split}
\end{align*}
and
\begin{align*}
\begin{split}
\left \{
\begin{array}{ll}
\Int^{\q}(\wgi, \lwe_{-\frac{2}{1}})&=1+\qz^{-1}\qx,\\
\Int^{\q}(\wgz, \lwe_{-\frac{2}{1}})&=\qz^{-3}\qx^2.
\end{array}
\right.
\end{split}
\end{align*}
By \Cref{thm:rightint}, we have
\begin{equation*}
\Big[\frac{3}{2}\Big]^{\sharp}_{\q}=\frac{\Int^{\q}(\wgi, \lwe_{\frac{3}{2}})}{\Int^{\q}(\wgz, \lwe_{\frac{3}{2}})}\bigg|_{\q=\qz^{-1}\qx}=\frac{\q^2+\q+1}{\q+1},
\end{equation*}
and
\begin{equation*}
\Big[-\frac{2}{1}\Big]^{\sharp}_{\q}=-\frac{\qz^{-1}\Int^{\q}(\wgi, \lwe_{-\frac{2}{1}})}{\Int^{\q}(\wgz, \lwe_{-\frac{2}{1}})}\bigg|_{\q=\qz^{-1}\qx}=-\frac{\q+1}{\q^2}.
\end{equation*}
\end{example}
\subsection{Combinatorial properties via \txq-intersections}\label{sec:otherprop}
\def\di{z_{\infi}}
\def\dii{z_*}
\def\diii{z_{\zero}}
Next, we give some topological explanation of some properties in \cite{MGO20} via \txq-intersections. We draw $\ug_{\zero}$ and $\ug_{\infi}$ as foliations in the branched double cover $\bc$, which intersect $C_{\zero}$ and $C_{\infi}$ only once respectively (see Figure \ref{fig:fundom}). We have the following corollary.
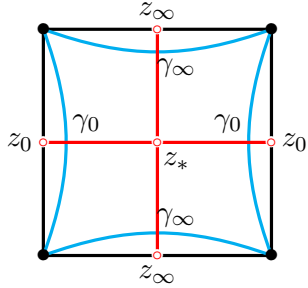
\begin{figure}[ht]\centering
\begin{tikzpicture}[scale=.5]
\draw[very thick](0,0)--(6,0)--(6,6)--(0,6)--(0,0);
\draw[cyan, very thick](0,6)to[out=-20,in=200](6,6)to[out=250,in=110](6,0)to[out=160,in=20](0,0)to[out=70,in=290](0,6);
\draw (3,0)node[below]{$\di$} (0,3)node[left]{$\diii$}(3,6)node[above]{$\di$}(6,3)node[right]{$\diii$}(3.5,3)node[below]{$\dii$};
\draw[red, very thick] (3,0)--(3,6) (0,3)--(6,3);
\draw (3.5,0.5)node[above]{$\ug_{\infi}$} (0.5,3.5)node[right]{$\ug_{\zero}$}(3.5,5.5)node[below]{$\ug_{\infi}$}(5.5,3.5)node[left]{$\ug_{\zero}$};
\draw (0,0)node{$\bullet$} (6,0)node{$\bullet$}(0,6)node{$\bullet$}(6,6)node{$\bullet$};
\draw[white] (3,0)node{$\bullet$} (6,3)node{$\bullet$}(0,3)node{$\bullet$}(3,6)node{$\bullet$} (3,3)node{$\bullet$};
\draw[red](3,0)\ww (0,3)\ww(3,6)\ww(6,3)\ww(3,3)\ww;
	\end{tikzpicture}
\caption{The branched double cover $\bc$ of $\disk$ and its foliations}
\label{fig:fundom}
\end{figure}

\begin{corollary}[{\cite{MGO20}}]\label{cor:comp}
For any rational number $\frac{r}{s}\in\upQ$, its right and left \txq-deformations satisfy the following properties.
\begin{description}
\item[Positivity] The polynomials \emph{$\qr^{\sharp}(r/s), \qs^{\sharp}(r/s), \qr^{\flat}(r/s)$} and  \emph{$\qs^{\flat}(r/s)$} have positive integer coefficients.
\item[Specialization] If we take $\q=1$, we have
\emph{
\begin{align}
\begin{split}
\left \{
\begin{array}{ll}
\qr^{\sharp}(r/s)|_{\q=1}&=\qr^{\flat}(r/s)|_{\q=1}=r,\\
\qs^{\sharp}(r/s)|_{\q=1}&=\qs^{\flat}(r/s)|_{\q=1}=s.
\end{array}
\right.
\end{split}
\end{align}}
\end{description}
\end{corollary}
\begin{proof}
The positivity follows from Theorem \ref{thm:rightint}, Theorem \ref{thm:leftint} and that intersection numbers are all positive.

For specialization, we consider the branched double covering $\bc$ of $\disk$. Then the closed arc $\ue_{\frac{r}{s}}$ becomes the simple closed curve $C_{\frac{r}{s}}$, whose preimage under $\widetilde{\pi}$ is a line with slope $\frac{r}{s}$, on $\bc$ through these corresponding decorations which are endpoints of $\ue_{\frac{r}{s}}$. When we take $\qz=\qx=1,$ the $\q$-intersection degenerates to usual intersection. By construction above, we have
\begin{align*}
\Int_{\disk}(\ue_{\frac{r}{s}},\ue_{\zero})=\frac{1}{2}\cdot\Int_{\bc}(C_{\frac{r}{s}},C_{\zero})=r,&\quad \Int_{\disk}(\ue_{\frac{r}{s}},\ue_{\infi})=\frac{1}{2}\cdot\Int_{\bc}(C_{\frac{r}{s}},C_{\infi})=s;\\
\Int_{\disk}(\ug_{\infi},\ue_{\frac{r}{s}})=\frac{1}{2}\cdot\Int_{\bc}(\ug_{\infi},C_{\frac{r}{s}})=r,&\quad \Int_{\disk}(\ug_{\zero},\ue_{\frac{r}{s}})=\frac{1}{2}\cdot\Int_{\bc}(\ug_{\zero},C_{\frac{r}{s}})=s.
\end{align*}
Therefore, the results follows from Theorem \ref{thm:leftint} and \Cref{thm:rightint}.
\end{proof}
\begin{example}
We notice that in the example of $\frac{3}{s}$, $\ue_{\frac{3}{2}}$ hits $\ug_{\infi}$ three times and hits $\ug_{\zero}$ twice in $\surfo$, which implies that $\qr^{\sharp}(3/2)|_{\q=1}=3$ and  $\qs^{\sharp}(3/2)|_{\q=1}=2$.
\end{example}

\section{Categorification}\label{sec:cat}
\subsection{Ginzburg algebra and derived categories}
\begin{definition}[{\cite{K09,IQ1}}]\label{def:X-Gin}
Let $Q=(Q_0,Q_1)$ be a finite quiver with vertices set $Q_0=\{1,2,\ldots,n\}$ and arrows set $Q_1$. The \emph{Ginzburg Calabi-Yau-$\XX$ ddg algebra} $\Gamma_{\XX}Q:=(\k \overline{Q},d)$ is defined as follows. We define a $\ZZ\oplus\ZZ\XX$-graded quiver $\overline{Q}$ with the same vertices set as $Q_0$ and the following arrows:
\begin{itemize}
\item original arrows $a:i\to j\in Q_1$ with degree 0;
\item opposite arrows $a^*:j\to i\in Q_1$ associated to $a\in Q_1$ with degree $2-\XX$;
\item a loop $e_i^*$ for each $i\in Q_0$ with degree $1-\XX$, where $e_i$ is the idempotent at $i$.
\end{itemize}
Let $\k\overline{Q}$ be a $\ZZ\oplus\ZZ\XX$-graded path algebra of $\overline{Q}$, and define a differential $d:\k\overline{Q}\to\k\overline{Q}$ of degree 1 by
\begin{itemize}
\item $da=da^*=0$ for $a\in Q_1$;
\item $de^*_i=e_i\big(\sum_{a\in Q_1}(aa^*-a^*a)\big)e_i$.
\end{itemize}
\end{definition}
We denote by $\D_{\XX}(Q):=\pvd\Gamma_{\XX}Q$ the perfect value derived category of $\Gamma_{\XX}Q$, which is the same as the finite-dimensional derived category of $\Gamma_{\XX}Q$.
We consider the $A_2$ case, where $A_2$ is a quiver with vertices set $\{1,2\}$ and an arrow $1\to 2$, and the corresponding category is denoted by $\cat{\XX}$.
\subsection{Rational case via spherical objects}
In this subsection, we aim to find spherical objects in some certain category which correspond to rational numbers and represent their hom space via right and left \txq-deformations. We particularly consider the case of the Calabi-Yau-$\XX$ category of the $A_2$ quiver. Recall that a triangulated category $\D$ is called \emph{Calabi-Yau-$\XX$} if for any objects $L, M$ in $\D$, we have a natural isomorphism
$$\Hom_{\D}(L,M)\cong D\Hom_{\D}(M, L[\XX]),$$
where $D=\Hom(-,\k)$ is the dual functor and $\k$ is an algebraically closed field. In particular, $\cat{\XX}$ is a Calabi-Yau-$\XX$ category with a distinguish auto-equivalence
$$\XX: \cat{\XX}\to\cat{\XX}.$$
\begin{definition}[{\cite{IQZ}}]
For any $M,N\in\D$, we define the \emph{bigraded Hom} as $$\Hom^{\ZZ^2}(M,N):=\bigoplus_{\ii,\sii\in\ZZ}\Hom_{\D}(M,N[\ii+\sii\XX]),$$
	and its \txq-dimension as
	\begin{gather}\label{eq:qHom}
	\qdH(M,N)\colon=\sum_{\ii,\sii\in\mathbb{Z}}
	\qz^{\ii}\qx^{\sii}\cdot\dim\Hom_{\D}(M,N[\ii+\sii\XX]).
	\end{gather}
	When $M=N$, $\Hom^{\ZZ^2}(M,M)$ becomes a $\ZZ^2$-graded algebra, called the Ext-algebra of $M$ and denoted by $\Ext^{\ZZ^2}(M,M)$.
\end{definition}
By definition, we directly have
\begin{equation}
\dim\Hom^{\ZZ^2}(M,N)=\qdH(M,N)\mid_{\qz=\qx=1}.
\end{equation}
\begin{definition}[{\cite{IQZ}}]
An object $S$ is called \emph{$\XX$-spherical} if $\Hom^{\bullet}(S,S)=\k\oplus\k[-\XX]$.
\end{definition}
For any spherical object $S$ in an Calabi-Yau-$\XX$ category $\D$, there is an associated auto-equivalence, namely the \emph{twist functor} $\phi_S:\D\rightarrow\D$, defined by
$$\phi_S(X)=\Cone(S\otimes\Hom^{\bullet}(S,X)\rightarrow X)$$
with inverse
$$\phi_S^{-1}(X)=\Cone(X\rightarrow S\otimes\Hom^{-\bullet}(X,S))[-1].$$
By \cite[Lemma 2.11]{ST}, we have the formula
\begin{equation}
\phi_{\psi(M)}=\psi\circ\phi_M\circ\psi^{-1}
\end{equation}
for any spherical object $M$ and any automorphism $\psi$ in $\Aut\D$. We define $\qSph(\DGA)$ to be the set of all spherical objects in $\cat{\XX}$ which are simples in some hearts (cf. \cite[Section 10]{KQ}). Let
\[\Sph(\DGA):=\qSph(\DGA)/\<[1],[\XX]\>.\]
We define $\ST(A_2)$ to be the subgroup of $\Aut\cat{\XX}$ generated by $\phi_{S}$ for any $S\in\qSph(\DGA)$.

Let $\disk$ be the three decorated disk as before.
There are reachable spherical objects up to shifts $[1]$ and $\XX$ in $\cat{\XX}$
corresponding to rational numbers. We have a categorification of \Cref{cor:biCA} as follows.
\begin{proposition}[{\cite[\S 4]{IQZ}}]\label{prop:biject}
There is a bijection $X$ and an isomorphism $\iota$ which fits into the following:
\begin{equation}
\begin{tikzpicture}[xscale=.6,yscale=.6]
\draw(180:3)node(o){$\wXCA(\disk)$}(-3,2.2)node(b){\small{$\BT(\disk)$}}
(0,2.5)node{$\iota$}(0,.5)node{$X$};
\draw(0:3)node(a){$\qSph(\DGA)$}(3,2.2)node(s){\small{$\ST(A_2)$}};
\draw[->,>=stealth](o)to(a);\draw[->,>=stealth](b)to(s);
\draw[->,>=stealth](-3.2,.6).. controls +(135:2) and +(45:2) ..(-3+.2,.6);
\draw[->,>=stealth](3-.2,.6).. controls +(135:2) and +(45:2) ..(3+.2,.6);
\end{tikzpicture}
\end{equation}
sending $\lwe_{\pm\frac{r}{s}}$ to $X_{\pm\frac{r}{s}}$ and $B_{\ue_{\pm\frac{r}{s}}}$ to $\phi_{X_{\pm\frac{r}{s}}}$, satisfying
\begin{equation}
X_{B_{\ue_{\frac{u}{v}}}(\lwp)}=\phi_{X_{\frac{u}{v}}}(X_{\frac{p}{q}}),
\end{equation}
and
\begin{equation}
X_{B_{\ue_{-\frac{p}{q}}}(\lwe_{-\frac{u}{v}})}=\phi_{X_{-\frac{p}{q}}}(X_{-\frac{u}{v}})
\end{equation}
for $\frac{r}{s}\in\upQ$.
Hence, \eqref{eq:ind} translates to the triangle
\begin{equation}\label{eq:trian}
\begin{tikzcd}
\Xp \ar[r] & \Xqs\ar[r]&\Xqq[(l+1)(1-\XX)]\ar[r]& \Xp[1],
\end{tikzcd}
\end{equation}
where $l=l(\frac{p}{q}\oplus\frac{u}{v})$ is the integer in \Cref{set}.
\end{proposition}
In fact, when we draw $\{X_{\frac{r}{s}}\}_{\frac{r}{s}\in\upQ}$ on the weighted Farey graph in the right one in Figure \ref{cat}, $\qd^l, l\geq 0$ represents that there is a morphism of degree $l(1-\XX)$ between the connected spherical objects.
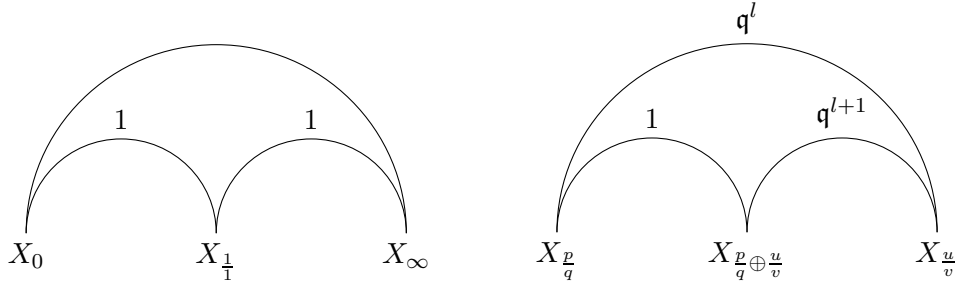
\begin{figure}
\begin{center}
\begin{tikzpicture}[scale=5]
    \draw (1,0) arc (0:180:.5);
    \draw  (0,0)  node[below]{$\displaystyle \Xze$};
    \draw  (1,0)  node[below]{$\displaystyle \Xin$};
    \draw (1,0) arc (0:180:.25);
    \draw  (.5,0) node[below]{$\displaystyle \Xon$};
    \draw (.5,0) arc (0:180:.25);
\draw (.25,.25) node[above]{$1$};
\draw (.75,.25) node[above]{$1$};
\end{tikzpicture}\hspace{1cm}
\begin{tikzpicture}[scale=5]
   \draw (1,0) arc (0:180:.5);
    \draw  (0,0)  node[below]{$\displaystyle \Xp$};
    \draw  (1,0)  node[below]{$\displaystyle \Xqq$};
    \draw (1,0) arc (0:180:.25);
    \draw  (.5,0) node[below]{$\displaystyle \Xqs$};
    \draw (.5,0) arc (0:180:.25);
\draw (.25,.25) node[above]{$1$};
\draw (.75,.25) node[above]{$\qd^{l+1}$};
\draw (.5,.5) node[above]{$\qd^{l}$};
\end{tikzpicture}
\end{center}
\caption{The categorification}
    \label{cat}
\end{figure}

Here is the iterative construction of \Cref{prop:biject}. Let $\Xze, \Xin$ be two spherical objects which are simples in some canonical heart with $\Ext^1(\Xin,\Xze)\neq 0$. Let $\Xon=\phi_{\Xin}(\Xze)$, and we deduce that $\Xon$ is also a spherical object. We have a triangle
$$\begin{tikzcd}
\Xze \ar[r] & \Xon\ar[r]&\Xin\ar[r]& \Xze[1]
\end{tikzcd}$$
by construction. We assume that the triangle in \eqref{eq:trian} holds for $l$ and consider the $l+1$ case. By Calabi-Yau-$\XX$ duality, we have non-zero morphisms $\Xqs\rightarrow \Xp[\XX]$ and $X_{\frac{u}{v}}\rightarrow \Xqs[(l+1)\XX-l]$. Hence we can extend them to triangles:
$$\begin{tikzcd}\Xp\ar[r]&Y\ar[r]&\Xqs[1-\XX]\ar[r]&\Xp[1]\end{tikzcd}$$
and
$$\begin{tikzcd}\Xqs\ar[r]&Y'\ar[r]&X_{\frac{u}{v}}[(l+1)(1-\XX)]\ar[r]&\Xqs[1],\end{tikzcd}$$
where these new spherical objects $Y=\phi_{\Xqs[1-\XX]}(\Xp)=\phi_{\Xqs}(\Xp)$ and $Y'=\phi_{X_{\frac{u}{v}}[(l+1)(1-\XX)]}(\Xqs)=\phi_{X_{\frac{u}{v}}}(\Xqs)$. Thus we construct the spherical objects associated to non negative rational numbers and the negative case is similar.
\begin{theorem}[{\cite[Lemma 3.4, Proposition 4.6]{IQZ}}]\label{thm:intform}
For any $\lwe,\lwe'\in\wXCA(\disk)$ satisfying $\Int_{\diski}(\lwe,\lwe')=0$, and $\wg_i\in\A$, we have
\begin{gather}
  \qdH(P_i,X_{\lwe})=\qqInt(\wg_i,\lwe),
\end{gather}
and
\begin{gather}\qdH(X_{\lwe},X_{\lwe'})=\qqInt(\lwe,\lwe'),\end{gather}
where $P_i$ is the indecomposable projective module corresponding to $\wg_i$.
\end{theorem}
By \Cref{thm:leftint}, \Cref{thm:rightint} and \Cref{thm:intform}, we have the following corollaries directly.
\begin{corollary}\label{cor:lefthom}
For any rational number $\frac{r}{s}\in\QQ\setminus\{0\}$, the fraction
\begin{equation}\label{eq:lefthom}
\frac{\varepsilon\qdH(\Xq,\Xze)}{\qdH(\Xq,\Xin)}\bigg|_{\q=\qz^{-1}\qx}
\end{equation}
corresponds to the left \txq-deformation of $\rs$, where
\begin{equation*}
\varepsilon=
\begin{cases}
\qz^{-1},&\frac{r}{s}\geq 0,\\
-1,&\frac{r}{s}<0.
\end{cases}
\end{equation*}
\end{corollary}
\begin{corollary}\label{cor:righthom}
For any rational number $\frac{r}{s}\in\uQ$, the fraction
\begin{equation}\label{eq:righthom}
\frac{\varepsilon\qdH(P_{\infi},\Xq)}{\qdH(P_{\zero},\Xq)}\bigg|_{\q=\qz^{-1}\qx}
\end{equation}
corresponds to the right \txq-deformation of $\rs$, where $P_{\zero}$ and $P_{\infi}$ are the corresponding indecomposable projectives satisfying that $\qdH(P_i,X_j)=\delta_{i,j}, i,j\in\{0,\infty\}$. Here
\begin{equation*}
\varepsilon=
\begin{cases}
1,&\frac{r}{s}\geq0,\\
-\qz^{-1},&\frac{r}{s}<0.
\end{cases}
\end{equation*}
\end{corollary}
\section{Applications}
\subsection{Reduction to single grading as foliations}
Let $N\geq 2$ be an integer. We collapse the double grading $\Lambda$ on $\surfo$ to single grading $\lambda$, which is a line field (or foliation) in $\PTSt$ by setting $\XX=N$.
More precisely, a double grading $(a,b)\in\ZZ\oplus\ZZ\XX$ collapses into $a+bN\in\ZZ$.
The foliations in such cases are given by quadratic differentials
\[
    (z^3+a z+b)^{N-2} \mathrm{d}z^{\otimes 2}
\]
on $\CC\PP^1$ with real blow-up at $\infty$, cf. Figure \ref{fig:folCY3} for $N=2,3,4$ and $\surfo=\disk$. Notice that the foliations come from quadratic differential (cf. \cite{I,BQS}).
Then $\qx=\qz^{N}$ and the \txq-intersection formula \eqref{eq:q-int} reduces to
\begin{equation}
	\qqInt(\lws,\lwt)=
	\displaystyle\sum_{k\in\mathbb{Z}}
	\qz^{k}\cdot\Int_{\Tri}^{k} (\lws,\lwt)
+(1+\qz^{N-1})
	\displaystyle\sum_{k\in\mathbb{Z}}
	\qz^{k}\cdot\Int_{\surfoi}^{k} (\lws,\lwt).
	\end{equation}
for $\lws,\lwt\in\wXCA(\surfo)$. Moreover, $\q=\qz^{N-1}$ in \Cref{thm:leftint} and \Cref{thm:rightint}.
When $N=2$, $\q=\qz$ and no specialization is required.
\begin{figure}[ht]\centering
\makebox[\textwidth][c]{
\begin{tikzpicture}[scale=.4,yscale=.8]
	\draw[very thick](0,0)circle(6);
 \foreach \j in {-57,-54,...,57} {
        \draw[Emerald!60, thin] plot [smooth,tension=1.7] coordinates {(0,6) (\j/10,0) (270:6)}; }
\draw[thick] plot [smooth,tension=1.7] coordinates {(0,6) (3,0) (270:6)};
\draw[thick] plot [smooth,tension=1.7] coordinates {(0,6) (-3,0) (270:6)};
\draw(0,6)node[cyan] {$\bullet$} (270:6)node[cyan] {$\bullet$} ;
 \draw(-3,0)edge[red,ultra thick](0,0);
\draw(0,0)edge[red,ultra thick](3,0);
	\draw(-3,0)\ww(0,0)\ww(3,0)\ww;
	\end{tikzpicture}
\begin{tikzpicture}[scale=.5,yscale=.8]
\draw[very thick](0,0)circle(5);
    \path (18+72:5) coordinate (v2)
          (-120:5) coordinate (v1)
          (-2.5,0) coordinate (v3)
          (0,0) coordinate (v4);
  \foreach \j in {.1, .18, .26, .34, .42, .5,.58, .66, .74, .82, .9}
    {
      \path (v3)--(v4) coordinate[pos=\j] (m0);
     \draw[Emerald!60, thin] plot [smooth,tension=.3] coordinates {(v1)(m0)(v2)};
    }
\draw[thick](v4)to(v1)to(v3)to(v2)to(v4);

    \path (18+72:5) coordinate (v2)
          (-60:5) coordinate (v1)
          (2.5,0) coordinate (v3);
  \foreach \j in {.1, .18, .26, .34, .42, .5,.58, .66, .74, .82, .9}
    {
      \path (v3)--(v4) coordinate[pos=\j] (m0);
      \draw[Emerald!60, thin] plot [smooth,tension=.3] coordinates {(v1)(m0)(v2)};
    }
\draw[thick](v4)to(v1)to(v3)to(v2)to(v4);
    \path (18+72:5) coordinate (v2)
          (18+72*2:5) coordinate (v1)
          (-2.5,0) coordinate (v3) (72+54:5) coordinate (v5);
    \path (v1)--(v2) coordinate[pos=.5] (v4);
  \foreach \j in {.2,.32,.45,.55,.68,.8, .9}
    {
      \path (v3)--(v5) coordinate[pos=\j] (m0);
      \draw[Emerald!60, thin] plot [smooth,tension=.6] coordinates {(v1)(m0)(v2)};
    }
    \path (-120:5) coordinate (v2);
    \path (v1)--(v2) coordinate[pos=.5] (v4) (72*2+54:5) coordinate (v5);
  \foreach \j in {.2,.32,.45,.55,.68,.8, .9}
    {
      \path (v3)--(v5) coordinate[pos=\j] (m0);
      \draw[Emerald!60, thin] plot [smooth,tension=.6] coordinates {(v1)(m0)(v2)};
    }
    \path (18+72:5) coordinate (v2)
          (18+72*0:5) coordinate (v1)
          (2.5,0) coordinate (v3) (58:5) coordinate (v5);
    \path (v1)--(v2) coordinate[pos=.5] (v4);
  \foreach \j in {.2,.32,.45,.55,.68,.8, .9}
    {
      \path (v3)--(v5) coordinate[pos=\j] (m0);
      \draw[Emerald!60, thin] plot [smooth,tension=.6] coordinates {(v1)(m0)(v2)};
    }
    \path (-60:5) coordinate (v2) (-22:5) coordinate (v5);
    \path (v1)--(v2) coordinate[pos=.5] (v4);
  \foreach \j in {.2,.32,.45,.55,.68,.8, .9}
    {
      \path (v3)--(v5) coordinate[pos=\j] (m0);
      \draw[Emerald!60, thin] plot [smooth,tension=.6] coordinates {(v1)(m0)(v2)};
    }
    \path (-120:5) coordinate (v2)
          (-60:5) coordinate (v1)
          (0,0) coordinate (v3) (270:5) coordinate (v5);
    \path (v1)--(v2) coordinate[pos=.5] (v4);
  \foreach \j in {.2,.3,.4,.5,.6,.7,.8,.9}
    {
      \path (v3)--(v5) coordinate[pos=\j] (m0);
      \draw[Emerald!60, thin] plot [smooth,tension=.5] coordinates {(v1)(m0)(v2)};
    }
\draw[thick](18+72*2:5)to(-2.5,0)(18:5)to(2.5,0);
\foreach \j in {1,2,5}
    {\draw[cyan](18+72*\j:5)\nn;}
\draw[cyan](-60:5)\nn(-120:5)\nn;
\draw[red,ultra thick]
  (-2.5,0)\ww to  (0,0)\ww to  (2.5,0)\ww ;
\end{tikzpicture}
\begin{tikzpicture}[scale=.5,yscale=.8]
\draw[very thick](0,0)circle(5);
    \path (-1.25,4.84) coordinate (v2)
          (-1.25,-4.84) coordinate (v1)
          (-2.5,0) coordinate (v3)
          (0,0) coordinate (v4);
  \foreach \j in {.1, .18, .26, .34, .42, .5,.58, .66, .74, .82, .9}
    {
      \path (v3)--(v4) coordinate[pos=\j] (m0);
     \draw[Emerald!60, thin] plot [smooth,tension=.3] coordinates {(v1)(m0)(v2)};
    }
\draw[thick](v4)to(v1)to(v3)to(v2)to(v4);

    \path (1.25,4.84) coordinate (v2)
          (1.25,-4.84) coordinate (v1)
          (2.5,0) coordinate (v3);
  \foreach \j in {.1, .18, .26, .34, .42, .5,.58, .66, .74, .82, .9}
    {
      \path (v3)--(v4) coordinate[pos=\j] (m0);
      \draw[Emerald!60, thin] plot [smooth,tension=.3] coordinates {(v1)(m0)(v2)};
    }
\draw[thick](v4)to(v1)to(v3)to(v2)to(v4);
    \path (-1.25,4.84) coordinate (v2)
          (-3.75,3.31) coordinate (v1)
          (-2.5,0) coordinate (v3) (-2.5,4.33) coordinate (v5);
    \path (v1)--(v2) coordinate[pos=.4] (v4);
  \foreach \j in {.2,.32,.45,.55,.68,.8, .9}
    {
      \path (v3)--(v5) coordinate[pos=\j] (m0);
      \draw[Emerald!60, thin] plot [smooth,tension=.4] coordinates {(v1)(m0)(v2)};
    }
\path (-3.75,3.31) coordinate (v2)
          (-3.75,-3.31) coordinate (v1)
          (-2.5,0) coordinate (v3) (180:5) coordinate (v5);
    \path (v1)--(v2) coordinate[pos=.5] (v4);
  \foreach \j in {.2,.32,.45,.55,.68,.8, .9}
    {
      \path (v3)--(v5) coordinate[pos=\j] (m0);
      \draw[Emerald!60, thin] plot [smooth,tension=1] coordinates {(v1)(m0)(v2)};
    }
    \path (-1.25,-4.84) coordinate (v2);
    \path (v1)--(v2) coordinate[pos=.4] (v4) (-2.5,-4.33) coordinate (v5);
  \foreach \j in {.2,.32,.45,.55,.68,.8, .9}
    {
      \path (v3)--(v5) coordinate[pos=\j] (m0);
      \draw[Emerald!60, thin] plot [smooth,tension=.4] coordinates {(v1)(m0)(v2)};
    }
    \path (1.25,4.84) coordinate (v2)
          (3.75,3.31) coordinate (v1)
          (2.5,0) coordinate (v3) (2.5,4.33) coordinate (v5);
    \path (v1)--(v2) coordinate[pos=.4] (v4);
  \foreach \j in {.2,.32,.45,.55,.68,.8, .9}
    {
      \path (v3)--(v5) coordinate[pos=\j] (m0);
      \draw[Emerald!60, thin] plot [smooth,tension=.4] coordinates {(v1)(m0)(v2)};
    }
\path (3.75,3.31) coordinate (v2)
          (3.75,-3.31) coordinate (v1)
          (2.5,0) coordinate (v3) (0:5) coordinate (v5);
    \path (v1)--(v2) coordinate[pos=.4] (v4);
  \foreach \j in {.2,.32,.45,.55,.68,.8, .9}
    {
      \path (v3)--(v5) coordinate[pos=\j] (m0);
      \draw[Emerald!60, thin] plot [smooth,tension=1] coordinates {(v1)(m0)(v2)};
    }
    \path (1.25,-4.84) coordinate (v2);
    \path (v1)--(v2) coordinate[pos=.4] (v4) (2.5,-4.33) coordinate (v5);
  \foreach \j in {.2,.32,.45,.55,.68,.8, .9}
    {
      \path (v3)--(v5) coordinate[pos=\j] (m0);
      \draw[Emerald!60, thin] plot [smooth,tension=.4] coordinates {(v1)(m0)(v2)};
    }
    \path (-1.25,-4.84) coordinate (v2)
          (1.25,-4.84) coordinate (v1)
          (0,0) coordinate (v3) (270:5) coordinate (v5);
    \path (v1)--(v2) coordinate[pos=.5] (v4);
  \foreach \j in {.2,.3,.4,.5,.6,.7,.8,.9}
    {
      \path (v3)--(v5) coordinate[pos=\j] (m0);
      \draw[Emerald!60, thin] plot [smooth,tension=.4] coordinates {(v1)(m0)(v2)};
    }
\path (1.25,4.84) coordinate (v2)
          (-1.25,4.84) coordinate (v1)
          (0,0) coordinate (v3) (90:5) coordinate (v5);
    \path (v1)--(v2) coordinate[pos=.5] (v4);
  \foreach \j in {.2,.3,.4,.5,.6,.7,.8,.9}
    {
      \path (v3)--(v5) coordinate[pos=\j] (m0);
      \draw[Emerald!60, thin] plot [smooth,tension=.4] coordinates {(v1)(m0)(v2)};
    }
\draw[thick](-3.75,3.31)to(-2.5,0) (-3.75,-3.31)to(-2.5,0) (3.75,3.31)to(2.5,0)(3.75,-3.31)to(2.5,0);
\draw[cyan] (1.25,4.84)node{$\bullet$} (-1.25,4.84)node{$\bullet$}(1.25,-4.84)node{$\bullet$}(-1.25,-4.84)node{$\bullet$} (3.75,3.31)node{$\bullet$}(-3.75,3.31)node{$\bullet$}(3.75,-3.31)node{$\bullet$}(-3.75,-3.31)node{$\bullet$};
\draw[red,ultra thick]
  (-2.5,0)\ww to  (0,0)\ww to  (2.5,0)\ww ;
\end{tikzpicture}}
\caption{The foliations of the CY-2,3,4 case}
\label{fig:folCY3}
\end{figure}
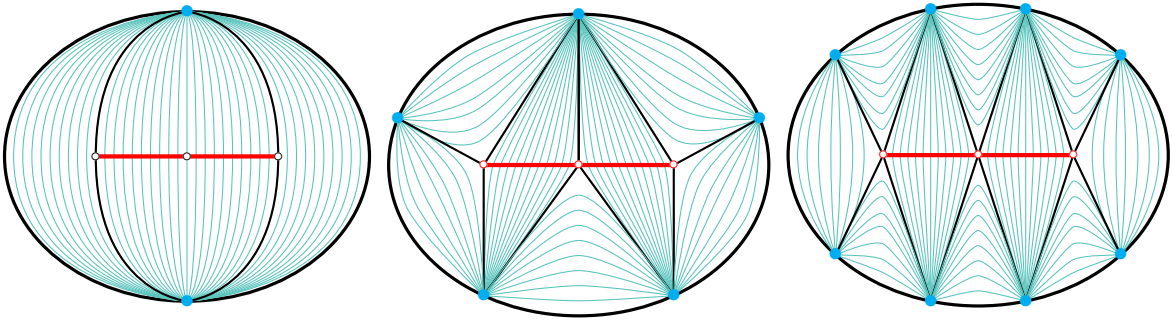

\subsection{Relation with BBL's results}
\def\DA{\Gamma_2A_2}
\def\nrsk{m(\frac{r}{s},k)}
In \Cref{def:X-Gin}, if we replace $\XX$ by an integer $N\geq 2$, we obtain $N$-Ginzburg dg algebra $\Gamma_N Q$ and the corresponding Calabi-Yau-$N$ category $\D_{N}(Q)$. That is, there is a projection
\begin{equation}
\pi_N: \Gamma_{\XX}Q\to\Gamma_N Q
\end{equation}
collapsing the double grading $\ZZ\oplus\ZZ\XX$ into $\ZZ$ by setting $\XX=N$ similar as above.
It induces a functor
$\pi_N:\D_{\XX}(Q)\to\D_{N}(Q).$
We consider the case when $Q=A_2$. For $\frac{r}{s}\in\upQ$, we claim that
\[\qdH(P_?,\Xq)=\sum_{k\in\ZZ}\nrsk\qz^{-k}\qx^k,\]
where $\nrsk$ is the occurrence times of $X_?[k-k\XX]$ in the HN-filtration of $\Xq$, $?\in\{0,\infty\}$. It holds naturally from the fact that
\begin{equation}\label{eq:delta}
\qdH(P_i, X_j[\ii+\sii\XX])=\delta_{i,j}\delta_{\ii,0}\delta_{\sii,0}
\end{equation}
where $i,j\in\{0,\infi\}$
and induction on $l(\frac{r}{s})$.

Consider the case when $N=2$ again with $\q=\qz$ and we write $\qz=q$.
In \cite{BBL}, they define two functionals
$$\occq, \uhom: \qSph^{\ZZ}(\DA)\times\qSph^{\ZZ}(\DA)\to\ZZ[q^{-1},q],$$
where $\qSph^{\ZZ}(\DA)$ is the set of spherical objects in $\D_2(A_2)$.
The former one $\occq(X_?, \Xq)$ counts the occurrences of $X_?$ in the HN-filtration of $\Xq$ for $?\in\{0,\infty\}$.
By \eqref{eq:delta}, we deduce that
\begin{equation*}
\occq(X_?,X)=\qdH(P_?,X)\mid_{\qz^{-1}\qx=q^{-1}}.
\end{equation*}
The latter one is
\begin{equation*}
\uhom(L,M):=
\begin{cases}
q^k(q^{-2}-q^{-1}),&M\cong L[k],\\
\sum_{k\in\ZZ}\dim\Hom(L,M[k])q^{-k},&\mbox{otherwise}.\end{cases}
\end{equation*}
Recall that in \Cref{def:neg}, the $q$-deformations for negative rational numbers are defined as:
\[\left[-\frac{r}{s}\right]_q^*:=-q^{-1}\left[\frac{r}{s}\right]_{q^{-1}}^*,\]
where $\frac{r}{s}\in\QQ^+\cup\{\infty\}$ and $*\in\{\sharp, \flat\}$.
\begin{corollary}
When specializing $\XX=2$,
the formulae \eqref{eq:righthom} and \eqref{eq:lefthom} in \Cref{cor:righthom} and \Cref{cor:lefthom}
become the formulae in \cite[Theorem 3.7, Theorem 3.8]{BBL} respectively.
Notice that the condition $X\geq 0$ corresponds to our $X_{-\frac{r}{s}}$ with $\frac{r}{s}\geq 0$.
\end{corollary}
\subsection{Grothendieck group interpretation}
\def\qqr{\textbf{R}_{\qz^{-1}\qx}}
\def\qqs{\textbf{S}_{\qz^{-1}\qx}}
Recall that $\cat{\XX}$ is a Calabi-Yau-$\XX$ category. The Grothendieck group $\Gro(\cat{\XX})$ admits a basis $\{[X_{\zero}], [X_{\infi}]\}$ and is a $\ZZ[\q^{\pm1}]$-module defined by the action
\begin{equation}
\qz^l\qx^k\cdot [E]:=[E[-l-k\XX]].
\end{equation}
 We have the following result.
\begin{proposition}\label{prop:Gro}
For any $\frac{r}{s}\in\uQ$, we have
\begin{equation}
[\Xq]=\qqr^{\sharp}(\frac{r}{s})[\Xze]+\qqs^{\sharp}(\frac{r}{s})[\Xin],
\end{equation}
where $\qqr^{\sharp}$ $($resp. $\qqs^{\sharp})$ is a polynomial of $\qz^{-1}\qz$ if we take $\q=\qz^{-1}\qz$ in $\qr^{\sharp}$ $($resp. $\qs^{\sharp})$.
\end{proposition}
\begin{proof}
We only prove the non-negative case by induction on $l(\frac{r}{s})$. For the initial case, it holds for $\Xze$ and $\Xin$ obviously. We assume that it holds for $\Xp$ and $\Xqq$ where $\frac{p}{q},\frac{u}{v}$ are in \Cref{set}. For $\Xqs$, we have
\begin{equation}
\begin{split}
[\Xqs]&=[\Xp]+[\Xqq[(l+1)(1-\XX)]]\\
&=[\Xp]+\qz^{-l-1}\qx^{l+1}[\Xqq]\\
&=\qqr^{\sharp}(\frac{p}{q})[\Xze]+\qqs^{\sharp}(\frac{p}{q})[\Xin]+\qz^{-l-1}\qx^{l+1}\big(\qqr^{\sharp}(\frac{u}{v})[\Xze]+\qqs^{\sharp}(\frac{u}{v})[\Xin]\big)\\
&=\qqr^{\sharp}(\frac{p}{q}\oplus\frac{u}{v})[\Xze]+\qqs^{\sharp}(\frac{p}{q}\oplus\frac{u}{v})[\Xin],
\end{split}
\end{equation}
which implies the result.
\end{proof}
\subsection{Relation to Jones polynomials for rational case}
For every rational number $\frac{r}{s}>1$, there is an associated rational (two-bridge) knot $C(\frac{r}{s})$ in the sense of \cite{LS}. The \emph{Jones polynomial} of $C(\frac{r}{s})$ is defined via the skein relation in \cite{LS}, which is denoted by
\[V_{\frac{r}{s}}(t)\in t^{\frac{1}{2}}\ZZ[t, t^{-1}]\cup\ZZ[t, t^{-1}].\]
When we take $\q=-t^{-1}$ and multiply the leading term, we get a polynomial $J_{\frac{r}{s}}(\q)\in\ZZ[\q]$, which is called the \emph{normalized Jones polynomial}. There is a corollary directly from \cite[Theorem A.3]{BBL} and \Cref{thm:leftint}.
\begin{corollary}
For every rational $\frac{r}{s}>1$, we have $J_{\frac{r}{s}}(\q)=|\qqInt(\lwq,\lwez)\mid_{\q=\qz^{-1}\qx}|$, where $|\cdot|$ is the normalized one.
\end{corollary}

\end{document}